\documentclass[12pt,a4paper]{article}

\usepackage{graphicx,bm,url,microtype,color}
\usepackage{amsfonts,amssymb,amsthm,amsmath}
\usepackage[english]{babel}
\usepackage[a4paper,text={16cm,23.5cm},centering]{geometry}
\usepackage[compact,small]{titlesec}
\usepackage[utf8]{inputenc}
\usepackage[bookmarks=false]{hyperref}
\usepackage{multirow}
\usepackage[capitalise,noabbrev]{cleveref} 

\newcommand{\dif}{\text{\rm d}} 
\newcommand{\Emp}{\mathbb{E}}
\newcommand{\E}{\text{\rm E}}
\newcommand{\Prob}{\text{\rm P}} 

\newcommand{\B}{\mathbb{B}}
\newcommand{\sgn}{\mbox{\rm sgn}}
\newcommand{\R}{\ensuremath{\mathbb{R}}}
\newcommand{\cd}{\ensuremath{\longrightarrow_{\text{\rm d}}}}
\newcommand{\cp}{\ensuremath{ \longrightarrow_{\text{P}}}}

\newcommand{\aequiv}{\ensuremath{ \overset{L^1}{\sim}}}
\newcommand{\cw}{\ensuremath{\overset{L^1}{\longrightarrow}_{\text{\rm{w}}}}}
\theoremstyle{definition}

\theoremstyle{theorem}

\newtheorem{theorem}{Theorem}
\newtheorem{corollary}{Corollary}
\newtheorem{lemma}{Lemma}

\newenvironment{keywords}
    {\vspace*{3mm}
    {\noindent{}\textit{Keywords\/:}}
        \nopagebreak\small}
        {}

\newenvironment{MSC}
    {\vspace*{3mm}
    {\noindent{}\textit{MSC\/:}}
        \nopagebreak\small}
        {}

\title{New distance measures for classifying X-ray astronomy data into stellar
classes\footnote{Research by A.B. and J.C. was supported by the Spanish MEyC grants MTM2013-44045-P and MTM2016-78751-P.
K.G. acknowledges the support from the Chandra ACIS Team contract SV4-74018 (G. Garmire \& L. Townsley, PIs), issued by the Chandra X-ray Center, which is operated by the Smithsonian Astrophysical Observatory on behalf of NASA under contract NAS8-03060.}}
\author{
Amparo Ba\'{\i}llo$^a$, Javier C\'{a}rcamo$^a$ and Konstantin Getman$^b$ \\
{\normalsize $^a$ Departamento de Matem\'{a}ticas, Universidad Aut\'{o}noma de Madrid,} \\ {\normalsize 28049 Madrid (Spain)} \\
{\normalsize $^b$ Department of Astronomy and Astrophysics, Pennsylvania State University,} \\ {\normalsize University Park PA 16802-6305 (U.S.A.)}
}
\date{}

\begin{document}

\maketitle

\begin{abstract}
The classification of the X-ray sources into classes (such as extragalactic sources, background stars, \ldots) is an essential task in astronomy. Typically, one of the classes corresponds to extragalactic radiation, whose photon emission behaviour is well characterized by a homogeneous Poisson process. We propose to use normalized versions of the Wasserstein and Zolotarev distances to quantify the deviation of the distribution of photon interarrival times from the exponential class. Our main motivation is the analysis of a massive dataset from X-ray astronomy obtained by the Chandra Orion Ultradeep Project (COUP). This project yielded
a large catalog of 1616 X-ray cosmic sources in the Orion Nebula region, with their series of photon arrival times and associated energies.
We consider the plug-in estimators of these metrics, determine their asymptotic distributions, and illustrate their finite-sample performance with a Monte Carlo study.
We estimate these metrics for each COUP source from three different classes. We conclude that our proposal provides a striking amount of information on the nature of the photon emitting sources. Further, these variables have the ability to identify X-ray sources wrongly catalogued before. As an appealing conclusion, we show that some sources, previously classified as extragalactic emissions, have a much higher probability of being young stars in Orion Nebula.
\end{abstract}


\begin{MSC}
Primary 60K35; secondary 62G20, 62N05.
\end{MSC}


\section{Introduction and motivation} \label{Section_Introduction}

An important initial step in the analysis of stellar populations is the classification of samples into different classes of sources (see \cite{Broos_et_al_11}). The definition of the classes (foreground stars, background stars, different types of pre-main-sequence stars, etc.) depends on the research project, but it is always of interest to identify extragalactic sources (see \cite{Broos_et_al_11}; \cite{Feigelson_etal05}). Frequently, the allocation has a degree of uncertainty, to the extent that some of the astronomical sources might remain unclassified (see \cite{Getman_etal05b}) or even wrongly catalogued.

X-ray astronomy deals with the detection and observation of astrophysical objects by means of the  properties of their X-ray emissions.
There are many astronomical sources of X-rays, such as galaxy clusters, black holes or different types of stars.
In X-ray astronomy, classification of the data (that is, the X-ray sources) is accomplished using all the information provided by source features such as its location and X-ray and infrared properties (see \cite{Broos_et_al_11}).
As X-radiation is blocked by the atmosphere of Earth, cosmic X-ray emissions can only be detected by space telescopes.

This article is motivated by a real dataset obtained as a result of Chandra Orion Ultradeep Project (COUP). It was fulfilled with one of the ``Great Observatories'' of NASA, the Chandra X-ray space telescope. Chandra was designed to observe X-ray emissions from high-energy regions of the space such as supernovas, black holes or star clusters as the Orion Nebula.

In this work, we focus on a massive collection of X-ray astronomical sources derived from a 2003 exposure of Chandra to the Orion Nebula region (\cite{Getman_etal05a}).
For each of the sources captured by Chandra, the photon arrival times and associated energies were collected during a nearly continuous observation period of almost 10 days. The majority of these X-ray sources have been classified into one of three groups: lightly-obscured and heavily-obscured low-mass young stars; and extragalactic sources. The X-ray classification of young stellar objects in star forming regions is in general a complicated task, where numerous source properties are used as features.

A very informative fact employed in this stellar classification is that, on the one hand, extragalactic radiation usually has a constant photon emission rate. This can be illustrated via the light curve of the astronomical source, a graph depicting its brightness (measured, e.g., by the photon count rate) over the course of the observation period. The light curves of COUP sources 111 (Figure~\ref{Figure.LightCurves} (a)) and 1304 (Figure~\ref{Figure.LightCurves} (b)), classified as extragalactic in \cite{Getman_etal05b} are examples of a constant photon arrival rate. In this case, the point process constituted by the photon arrival times is well-modeled by a homogeneous Poisson process. Thus, photon interarrival times from an extragalactic source should be close to an exponential distribution. On the other hand, young stars usually exhibit high-amplitude rapid variability (\cite{Wolk_et_al05}) and their photon arrivals are generally affected by flares so the corresponding interarrival distribution might deviate from the exponential one. As an example, Figure~\ref{Figure.LightCurves} (c) displays the light curve of COUP source 89, corresponding to a young star, where we can see a large flare at about 80 hours from the start of the observation period. In the astrophysical literature, there are different proposals to quantify photon emission variability in stellar X-ray sources (see \cite{Wolk_et_al05}).

\begin{figure}
\begin{center}
\includegraphics[width=\textwidth]{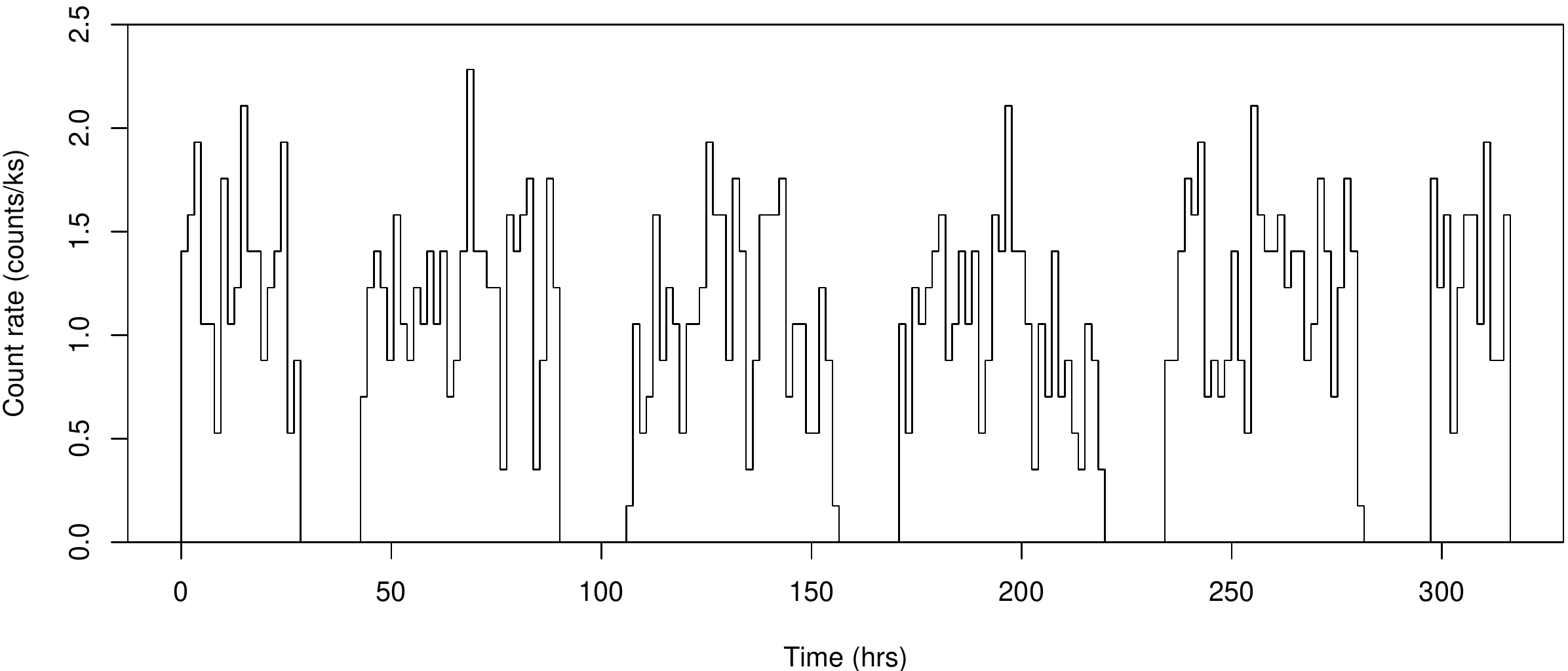} \\ (a) \\ [7 mm]
\includegraphics[width=\textwidth]{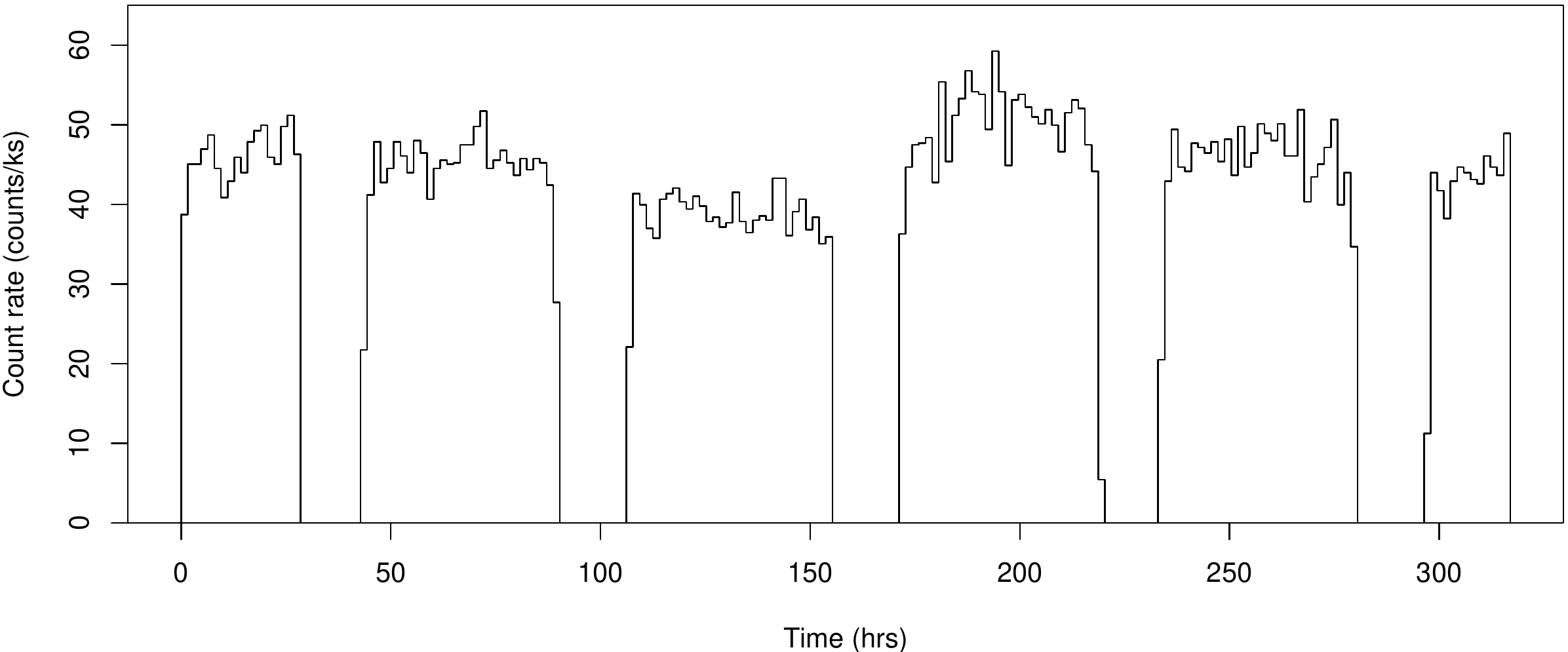} \\ (b) \\ [7 mm]
\includegraphics[width=\textwidth]{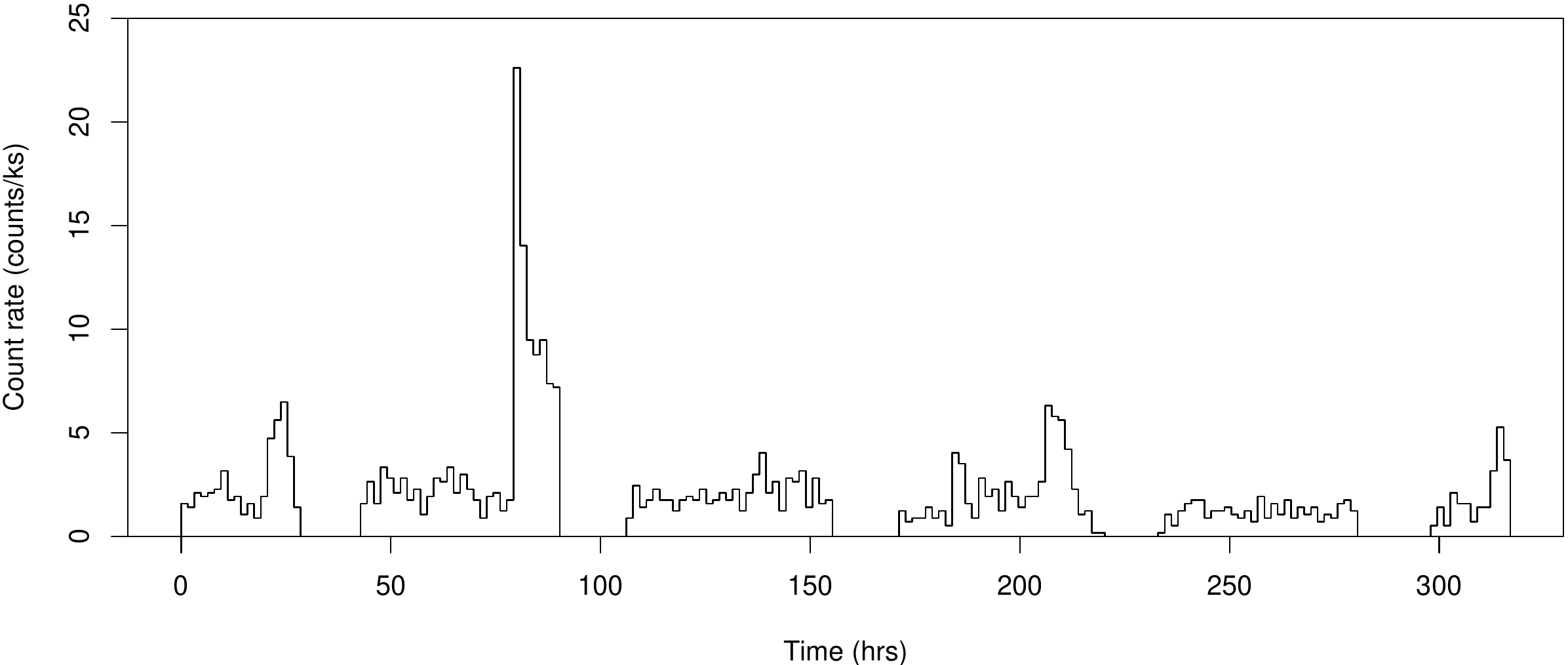} \\ (c)
\end{center}
\caption{Light curves for COUP sources (a) 111 (extragalactic radiation); (b) 1304 (extragalactic radiation); (c) 89 (lightly obscured PMS star).}
\label{Figure.LightCurves}
\end{figure}

Here, we propose a new statistical methodology to quantify the deviation of a random variable (namely, the photon interarrival time or PIT) from the exponential class.  The final aim is to generate a new input variable for the discriminant procedure distinguishing among the source classes, and particulary extragalactic ones (see the data analysis pipeline in Figure \ref{Figure.Pipeline}). We consider that a large estimated distance of the PIT to the exponential class is an evidence that the corresponding source is not extragalactic. Specifically, we use a normalized version of the so-called Wasserstein and Zolotarev $\zeta_2$ metrics, between the photon interarrival times of each X-ray source and the exponential distribution. As mentioned in \cite[Section 15]{Rachev}, the Zolotarev $\zeta_2$-metric is appropriate when dealing with exponential variables. Further, \cite{Rachev-2011} argue that Wasserstein and Zolotarev distances are more sensitive to extreme values  than other probability metrics such as the usual Kolmogorov distance. In general, it is often desirable to take into account extreme events to compare distributions. This is specially relevant with data coming from astrophysical studies (see \cite{FeigelsonBabu12}). We demonstrate that these distances can be used as informative variables that detect groups or similitudes among the X-ray sources and help identify possible outliers within a group. In this work the term ``outlier" refers to a source whose distance to the exponential distribution is substantially different from the distances of other members of the same class. 
In fact, in the final analysis of the COUP data, we show that some of the outlying COUP sources, initially classified as extragalactic radiation, could actually be young stars in the Orion Nebula region.

\begin{figure}
\begin{center}
\includegraphics[width=0.8\textwidth]{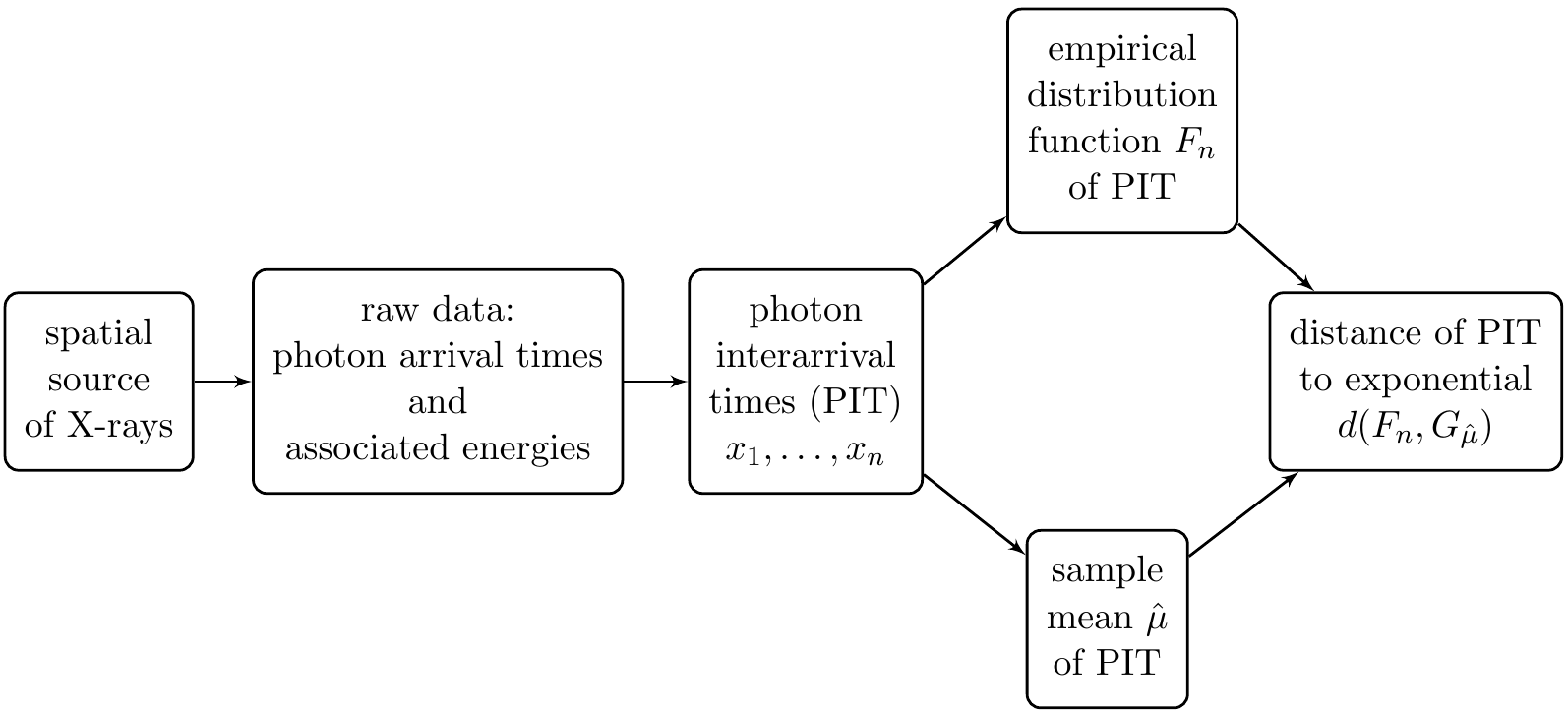}
\end{center}
\caption{Data analysis pipeline for each single cosmic X-ray source (input): the raw data is the series of photon arrival times and their corresponding energies. PIT are computed as differences between consecutive arrival times. The sample mean and the empirical distribution function of PIT produce the output (the empirical distance of the PIT distribution to the exponential one).}
\label{Figure.Pipeline}
\end{figure}

The ideas in this paper are also potentially useful in other biological or physical problems in which deviations from the exponential model need to be detected and quantified. Alternatively, the proposed methodology allows assessing whether a homogeneous Poisson process achieves a good approximation of an observed phenomenon.

This paper is structured as follows. In the next section, we describe in detail the COUP dataset. In Section \ref{Section Zolotarev}, we introduce the Wasserstein and Zolotarev distances and normalized versions of them. We also consider the plug-in estimators of these metrics to be used in practice and determine their asymptotic distributions. In Section \ref{Section.Simulations}, we carry out a simulation study to assess the practical performance of our proposal with finite samples. The COUP dataset is analyzed in depth in Section \ref{Section COUPAnalysis}. Finally, the main conclusions are collected in Section \ref{SectionConclusions}. In the Appendix, we include the proofs of the results stated in Section \ref{Section Zolotarev} and other technical details.


\section{Chandra Orion Ultradeep Project Dataset} \label{Section_COUP_Dataset}

Among other things, the COUP analyzes X-flaring in pre-main-sequence (PMS) stars, members of the Orion Nebula region that is composed of the rich revealed Orion Nebula Cluster (ONC) and the filamentary molecular cloud called Orion Molecular Cloud 1 (OMC-1). 
A PMS star is a premature star that has acquired all of its mass from its natal envelope of interstellar dust and gas and  contracts until it starts hydrogen burning (see, e.g., \cite{Schulz12}). These young stars have intense magnetic fields, detected through their X-ray emissions, where plasma, confined in magnetic loops, is heated to X-ray emitting temperatures.
Detection of these high-energy emissions is only possible by space observatories, such as the Chandra X-ray Observatory (see \url{http://chandra.si.edu/}).
The nearest rich and concentrated collection of PMS stars is in the ONC/OMC-1 star forming region.

In January 2003, the Chandra X-ray Observatory focused its Advanced CCD Imaging Spectrometer on the ONC for a period of 13.2 days obtaining the deepest X-ray observation ever taken of any star cluster (see \cite{Getman_etal05a}). This observation period was only interrupted by the five passages of the Chandra spacecraft through the Van Allen radiation belts. Graphically, this is reflected in the five gaps appearing in the light curves of Figure~\ref{Figure.LightCurves}. The results were an almost continuous observation of the photon arrival times and associated energies for 1616 X-ray sources. COUP sources were compared with source positions from previously existing catalogs, with the aim of physically associating the COUP sources with already identified stars (whenever this was possible). The majority of these COUP sources has been classified into one of three groups (see \cite{Feigelson_etal05}):
\begin{list}{--}{\leftmargin=0.5cm \itemindent=0em}
\item \textsl{Lightly-obscured} PMS sources: This class is constituted by 835 cool low-mass PMS stars that are likely located in the ONC cluster. The term ``lightly obscured'' means that the star X-ray emission is less absorbed by the material in the interstellar medium.
\item \textsl{Heavily-obscured} PMS sources: This class corresponds to 559 low-mass PMS stellar objects that are likely still embedded in the nascent OMC-1 cloud.
\item \textsl{Nonmembers}: This group contains over 200 probable nonmembers of the Orion Nebula star forming region. A large part are likely extragalactic sources and a few are foreground stars or very faint sources without counterparts. For our analysis, in this group we only consider the extragalactic X-ray emissions.
\end{list}

The original classification of these 1594 COUP sources is detailed in \cite{Getman_etal05b}, who provide an estimated probability of membership to the Orion cloud for each source. Thus, this classification has a degree of uncertainty.
The aim of this work is to introduce new distance measures that contain relevant information for classifying X-ray astronomy data into stellar classes. Specifically, we compute some probability metrics of the PIT to the exponential class, as they exhibit different distributions depending on the nature of the photon emission sources. These distances can be incorporated in any classification rule to reduce the  classification error.


\section{Quantifying the discrepancy from extragalactic radiation}\label{Section Zolotarev}

As mentioned in the introduction, PIT resulting from extragalactic radiation are usually well-modeled by exponential distributions while, for the PMS stars classes, PIT distributions might deviate from the exponential one.
Our primary objective is to compare different sources by computing some normalized probability metrics between their corresponding PIT and the nearest exponential variable, and to highlight the classification power of these features. This idea has been tackled before: for instance, in X-ray astronomy the usual Kolmogorov distance to the exponential class is often used in data classification. However, depending on the problem at hand, other distances could be more appropriate, e.g., to take into account extreme values or to highlight a specific part of the distribution. Here, we analyze in depth the performance of the Wasserstein and Zolotarev distances, more sensitive to the behaviour in the tail of the distribution than others such as the usual Kolmogorov and Cram\'{e}r-von Mises metrics. In particular, a byproduct of the results in this section is the possibility of testing for exponentiality (see the end of Subsection \ref{Section Zolotarev.Subsection Asymptotics}), for which there are numerous proposals in the literature (see the reviews by \cite{Ascher} and \cite{Henze-Meintanis}).


\subsection{The choice of the metrics} \label{Subsection_MetricsChoice}

A central question in the problem under consideration is the choice of a suitable metric to measure how far the probability distribution of the positive random variable $X$ of interest (the PIT),  with expectation $\mu>0$, is from the exponential variable $Y_\mu$ with the same mean.
In this work, we focus on the family of {integral probability metrics}
\begin{equation}\label{Zolotarev-metrics}
d_r(X,Y_\mu)=\sup \left\{|\E f(X)-\E f(Y_\mu)| :  f\in\mathcal{F}_r \right\}, \quad r\in \mathbb{N},
\end{equation}
where $\mathcal{F}_r$ is the class of real-valued functions $f$ on $\R$ having $r$-th derivative $f^{(r)}$ a.e. and such that $|f^{(r)}|\le 1$ a.e. Observe that $d_r(X,Y_\mu)$ is the maximum error in the expected value within the class of smooth functions $\mathcal{F}_r$ due to the approximation of $X$ by $Y_\mu$. For notational convenience, if $F$ and $G_\mu$ are the distribution functions of $X$ and $Y_\mu$, respectively, we indistinctly use $d_r(F,G_\mu)$ or $d_r(X,Y_\mu)$. Note also that $G_\mu(x)=1-\exp(-x/\mu)$, for $x\ge 0$.

The case $r=1$ in (\ref{Zolotarev-metrics}) has special relevance. By the Kantorovich--Rubinstein theorem, we see that $d_1\equiv \omega$ is the famous $L^1$-\emph{Wasserstein distance}.
For $r\ge 2$, $d_r\equiv\zeta_r$ is the \textit{Zolotarev metric} of order $r$. For a general reference on these distances we refer to \cite{Rachev}.

It can be proved that, when $d_r(X,Y_\mu)<\infty$, the moments of $X$ and $Y_\mu$ coincide up to order $r-1$. As a consequence, in  general, when $r\ge 3$ and $X$ is \textit{not} exponential, we have that $d_r(X,Y_\mu)=\infty$. This follows from the fact that, for many distributions, equalities $\E X=\mu$ and $\E X^2=2\,\mu^2$ are too restrictive and they actually imply that $X$ is exponential. For instance, this happens for the variables with the HNBUE or HNWUE property, two large families of random variables that include all the usual ageing and anti-ageing classes of distributions as it follows from results on stochastic equality under convex domination (see, e.g., \cite[Theorem 3.A.42, p. 133]{Shaked}). Therefore, only the cases $r=1$ and $r=2$ make sense for the discussed problem, that is, the Wasserstein distance) $d_1\equiv\omega$ and the Zolotarev metric $d_2\equiv\zeta_2$. These two metrics have easier-to-handle dual integral representations (see \cite{Rachev-2011}), given by
\begin{align}\label{Wasserstein representation}
\omega(F,G_\mu)&=\int_0^\infty |F(t)-G_\mu(t)|\,\dif t,\\[2 mm]
\zeta_2(F,G_\mu) &= \int_0^\infty \left| \int_t^\infty   (F(x)-G_\mu(x)) \, \dif x        \right|\, \dif t.\label{Zolotarev distance}
\end{align}

The metrics $\omega$ and $\zeta_2$ have practical advantages for the problem at hand. First, as argued in \cite[p. 15]{Rachev-2011}, they are more sensitive to the differences in the probabilities corresponding to extreme values than other common probability metrics such as the \textit{Kolmogorov distance}, $\kappa(F,G)=\sup_{x} |F(x)-G(x)|$.
Since the difference $|F(x)-G(x)|$ converges to zero as $x$ tends to $+\infty$ or $-\infty$, the contribution of the terms corresponding to extreme events is usually small. As a consequence, the differences in the tail behavior of $X$ and $Y$ will only be reflected in $\kappa(F,G)$ to a relatively small extent. However, representations (\ref{Wasserstein representation}) and (\ref{Zolotarev distance}) show that extreme values have more weight in $\omega$ and $\zeta_2$, as integrals of tail probabilities appear in these distances.
Additionally, Zolotarev $\zeta_2$-metric is considered as the ``natural metric" when dealing with the exponential class (see \cite[p. 340]{Rachev}).

It becomes soon apparent that these two metrics have the important drawback of not being scale-independent, an essential requirement to compare X-ray sources with very different photon emission rates (the average of photon emissions in each time interval). Even for sources of the same nature, the distances $\omega$ and $\zeta_2$ to the exponential distribution could be extremely different. This is clearly reflected in Figure~\ref{Figure.LightCurves}, where COUP 111 and 1304, both classified as extragalactic, would not be comparable without a suitable normalization.

To overcome this problem, next we introduce dimensionless versions of the $\omega$ and $\zeta_2$ distances.
Let us observe that the Wasserstein $\omega$ and Zolotarev $\zeta_2$ metrics are homogeneous of order 1 and 2, respectively. We define the \textit{normalized Wasserstein} and \textit{normalized Zolotarev} metrics as
$$
\bar\omega (X,Y_\mu)=\omega(X/\mu,Y_\mu/\mu) =\frac{1}{\mu}\omega(X, Y_\mu)
$$
and
$$
\bar\zeta_2 (X,Y_\mu)=\zeta_2(X/\mu,Y_\mu/\mu) =\frac{1}{\mu^2}\zeta_2(X, Y_\mu).
$$
These are the two probability distances that should be used in practice (instead of their unnormalized versions), as they are homogeneous of degree 0 and dimensionless.



\subsection{Estimation and large sample behaviour} \label{Section Zolotarev.Subsection Asymptotics}

In practice, the distribution of the random variable $X$ is usually unknown. Therefore, the distances $d(X,Y_\mu)$ (for $d=\omega, \zeta_2, \bar\omega$ and $\bar\zeta_2$) have to be estimated using a random sample $X_1,\dots,X_n$ from $X$. We propose to use the plug-in estimators obtained by replacing the true distribution $F$ of $X$ by the empirical distribution $F_n$ of the sample $X_1,\dots,X_n$, that is,
$F_n(t) = n^{-1} \sum_{i=1}^n I_{\{X_i\le t\}}$, $n\in \mathbb{N}$, $t\ge 0$,
where $I_A$ stands for the indicator function of the set $A$. Thus, we estimate $d(F,G_\mu)$ by $d(F_n,G_{\hat\mu})$, where $\hat\mu=\frac{1}{n}\sum_{i=1}^n X_i$ is the sample mean (and the maximum likelihood estimator of the rate parameter $\mu$ of an exponential model).

Analyzing the asymptotic behavior of the empirical distances $d(F_n,G_{\hat\mu})$ is an important issue to understand its performance and accuracy in practice, for instance, to assess whether an exponential model provides a reasonably good approximation of $X$. Besides, the asymptotic probability distribution potentially allows performing inference on $d(X,Y_\mu)$. For the considered distances, we note that
\begin{equation*}
d(F_n,G_{\hat\mu})= d(F,G_\mu) + \frac{1}{\sqrt{n}}\delta_n(d,F),
\end{equation*}
where $\delta_n(d,F)$ is the standardized version of the estimated distances
\begin{equation}\label{normalized and estimated distance}
\delta_n(d,F):=\sqrt{n}\left(   d(F_n,G_{\hat\mu})-d(F,G_\mu)    \right),\quad  n\in \mathbb{N}.
\end{equation}
Here we find conditions (as sharp as possible) on the random variable $X$ so that $\delta_n(d,F)$
converges in distribution as $n\to\infty$, and determine its weak limit, $\delta_\infty(d,F)$. In this way, we obtain that
\begin{equation*}
d(F_n,G_{\hat\mu})= d(F,G_\mu) + O_\Prob(1/\sqrt{n}), \quad \text{as } n\to\infty.
\end{equation*}

Though the detailed proofs of the asymptotic distribution of $\delta_n(d,F)$ are collected in the Appendix, we describe here in broad strokes the main ideas behind them. First, we note that
\begin{equation}\label{delta as functional}
\delta_n(d,F)=\rho_n(\mathbb{X}_{d,n},g_d),
\end{equation}
where $\rho_n:L^1\times L^1\to\R$ is the functional defined by
\begin{equation}\label{rhon}
\rho_n(f,g):=\Vert f+\sqrt{n} g \Vert_1 -\sqrt{n} \Vert g\Vert_1,\quad\text{for }f,\, g\in L^1,
\end{equation}
$\mathbb{X}_{d,n}$ are the stochastic processes given (for $t\ge 0$) by
\begin{equation}\label{processes X(d)}
\begin{split}
\mathbb{X}_{\omega,n}(t) &:=\sqrt{n}\left[    (F_n(t)-G_{\hat\mu}(t)) -(F(t) -G_\mu(t))  \right],\\[3 mm]
\mathbb{X}_{\bar\omega,n}(t) &:=\sqrt{n} \left[   \frac{1}{\hat\mu} (F_n(t)-G_{\hat\mu}(t)) - \frac{1}{\mu}(F(t)-G_{\mu}(t))  \right],\\[3 mm]
\mathbb{X}_{\zeta_2,n}(t) &:=\sqrt{n}\left[  \int_t^\infty (F_n(x)-G_{\hat\mu}(x))\,\dif x-     \int_t^\infty (F(x)-G_{\mu}(x))\,\dif x     \right],\\[3 mm]
\mathbb{X}_{\bar{\zeta}_2,n}(t) &:=\sqrt{n}\left[ \frac{1}{\hat\mu^2} \int_t^\infty (F_n(x)-G_{\hat\mu}(x))\,\dif x-   \frac{1}{\mu^2}   \int_t^\infty (F(x)-G_{\mu}(x))\,\dif x     \right],
\end{split}
\end{equation}
and $g_d$ are the (deterministic) functions defined by
\begin{align}
g_\omega(t)&:= F(t)-G_{\mu}(t), &  g_{\bar\omega}(t)&:=\frac{1}{\mu}(F(t)-G_{\mu}(t)),\label{funtions omega}\\[3 mm]
g_{\zeta_2}(t)&:=\int_t^\infty (F(x)-G_{\mu}(x))\,\dif x, & g_{\bar\zeta_2}(t)&:=\frac{1}{\mu^2}\int_t^\infty (F(x)-G_{\mu}(x))\,\dif x.\label{funtions zeta}
\end{align}
From (\ref{delta as functional}), we see that establishing the (weak) convergence in $L^1$ of the processes $\mathbb{X}_{d,n}$ in (\ref{processes X(d)}), combined with the continuity of the linking functional in (\ref{rhon}), immediately translates into the convergence in distribution of $\delta_n(d,F)$.

Before stating the main results, we need to introduce some definitions and notation. In the sequel, $\B_F:=\B\circ F$ is the \emph{$F$-Brownian bridge}, where $\B$ is a standard Brownian bridge on $[0,1]$, that is, $\B$ is a centered Gaussian process with covariance function $\gamma(s,t)=\min(s,t)-st$ and continuous paths, with probability 1.

We consider the \emph{Lorentz spaces} of positive random variables defined by
$\mathcal{L}^{2,1} :=\left\{  X :  \Lambda_{2,1}(X)<\infty \right\}$ and $\mathcal{L}^{4,2} :=\left\{  X :  \Lambda_{4,2}(X)<\infty \right\}$, where
\begin{equation*}
\Lambda_{2,1}(X):=\int_0^\infty \sqrt{\Prob(X>t)}\,  \dif t     \quad       \text{and}      \quad       \Lambda_{4,2}(X):=\int_0^\infty t\sqrt{\Prob(X>t)}\,  \dif t
\end{equation*}
(see \cite[p. 279]{Ledoux-Talagrand}). Conditions $\Lambda_{2,1}(X)<\infty$ and $\Lambda_{4,2}(X)<\infty$ are slightly stronger than $\E X^2<\infty$ and $\E X^4<\infty$, respectively (see \cite{Grafakos}).
Finally, $\mathcal{L}^p:=\{X:\E X^p<\infty\}$ ($p>0$) is the usual space of (positive) random variables with finite $p$-th moment.

The following two theorems characterize the asymptotic behavior of $\mathbb{X}_{d,n}$ in $L^1$. The results are sharp in the sense that we obtain the exact integrability condition on $X$ so that the processes converge in distribution in $L^1$ as $n\to\infty$. The symbol ``$\cw$" stands for the weak convergence of a sequence of random processes in the space $L^1$  as $n\to\infty$ (see the Appendix for the precise definition).

\begin{theorem}\label{Theorem SP Wasserstein}
Let $X$ be a positive random variable with expectation $\mu>0$. If $X \in\mathcal{L}^{4/3}$, the following assertions are equivalent:
\begin{enumerate}
\item[(a)] $X\in \mathcal{L}^{2,1}$.
\item[(b)] $\mathbb{X}_{\omega,n}\cw \mathbb{X}_{\omega,F}$, where $\mathbb{X}_{\omega,F}$ is a centered Gaussian process given by
\begin{equation}\label{XFW}
\mathbb{X}_{\omega,F}(t):=\B_F(t)- \frac{t}{\mu^2} e^{-t/\mu} \int_0^\infty \B_F(s)\,\dif s,\quad t\ge 0.
\end{equation}
\item[(c)] $\mathbb{X}_{\bar\omega,n}\cw \mathbb{X}_{\bar\omega,F}$, where $\mathbb{X}_{\bar\omega,F}$ is a centered Gaussian process given by
\begin{equation*}
\mathbb{X}_{\bar\omega,F}(t):=\frac{1}{\mu} \left[\B_F(t)+ \left(  g_{\bar\omega}(t)-\frac{t}{\mu^2}e^{-t/\mu}      \right) \int_0^\infty \B_F(s)\,\dif s\right],\quad t\ge 0,
\end{equation*}
and the function $g_{\bar\omega}$ is defined in (\ref{funtions omega}).
\end{enumerate}
\end{theorem}

\begin{theorem}\label{Theorem SP Zolotarev}
Let $X$ be a positive random variable with expectation $\mu>0$. The following assertions are equivalent:
\begin{enumerate}
\item[(a)] $X\in \mathcal{L}^{4,2}$.
\item[(b)] $\mathbb{X}_{\zeta_2,n}\cw \mathbb{X}_{\zeta_2,F}$, where $\mathbb{X}_{\zeta_2,F}$ is a centered Gaussian process given by
\begin{equation*}
\mathbb{X}_{\zeta_2,F}(t):=\int_t^\infty \B_F(s)\, \dif s- \left(1+\frac{t}{\mu} \right)e^{-t/\mu} \int_0^\infty \B_F(s)\,\dif s,\quad t\ge 0.
\end{equation*}
\item[(c)] $\mathbb{X}_{\bar\zeta_2,n}\cw \mathbb{X}_{\bar\zeta_2,F}$, where $\mathbb{X}_{\bar\zeta_2,F}$ is a centered Gaussian process given by
\begin{equation}\label{XFZB}
\mathbb{X}_{\bar\zeta_2,F}(t):=\frac{1}{\mu^2}\left[   \int_t^\infty \B_F(s)\, \dif s+ \left(2\,\mu\, g_{\bar\zeta_2}(t)- \left(1+\frac{t}{\mu}\right)e^{-t/\mu}\right) \int_0^\infty \B_F(s)\, \dif s\right],
\end{equation}
for $t\ge 0$, and the function $g_{\bar\zeta_2}$ is defined in (\ref{funtions zeta}).
\end{enumerate}
\end{theorem}


Using (\ref{delta as functional}) and Theorems~\ref{Theorem SP Wasserstein} and \ref{Theorem SP Zolotarev}, in the next theorem we derive the asymptotic distribution of $\delta_n(d,F)$. In the sequel ``$\cd$" stands for convergence in distribution as $n\to\infty$, $\sgn(\cdot)$ denotes the sign function and $A^c$ is the complement of the set $A$.

\begin{theorem}\label{Theorem 3}
Let $X$ be  a positive random variable with expectation $\mu>0$. For $d=\omega$ or $d=\bar\omega$ (respectively, for $d=\zeta_2$ or $d=\bar\zeta_2$), let us assume that $X\in \mathcal{L}^{2,1}$ (respectively, $X\in \mathcal{L}^{4,2}$). Then, $\delta_n(d,F) \cd \delta_\infty(d,F)$, with
\begin{equation}\label{delta-infinity}
\delta_\infty(d,F):=\int_{I({g_d})} |\mathbb{X}_{d,F}(t)| \,\dif t +\int_{{I({g_d})}^c} \mathbb{X}_{d,F}(t) \, \sgn(g_d(t)) \,\dif t,
\end{equation}
where the processes $\mathbb{X}_{d,F}$ are defined in (\ref{XFW})-(\ref{XFZB}), the functions $g_d$ are given in (\ref{funtions omega})-(\ref{funtions zeta}), and $I(g_d):=\{t\ge 0 : g_d(t)=0\}$.
\end{theorem}

The next corollary, a direct consequence of Theorem \ref{Theorem 3}, provides the asymptotic distribution of $\delta_n(d,F)$, when $F$ is an exponential distribution function. In such a case, $d(F,G_\mu)=0$ and the estimators behave as a random quantity at the order of $O_\Prob(1/\sqrt{n})$. It is interesting to note that the limiting distribution of the normalized distances does not depend on the unknown mean of the exponential distribution.

\begin{corollary}\label{Corollary exponential}
For the processes $\mathbb{X}_{d,F}$ defined in (\ref{XFW})-(\ref{XFZB}), if $X$ follows an exponential distribution with mean $\mu$, then
\begin{enumerate} \label{Corollary.Asymptotic.Exponential}
\item[(a)] $\sqrt{n} \, \omega(F_n,G_{\hat\mu}) \cd \Vert \mathbb{X}_{\omega,G_\mu} \Vert_1=\mu\, \Vert \mathbb{X}_{\omega,G_1} \Vert_1$;
\item[(b)] $\sqrt{n} \, \bar\omega(F_n,G_{\hat\mu}) \cd \Vert \mathbb{X}_{\bar\omega,G_1} \Vert_1 = \Vert \mathbb{X}_{\omega,G_1} \Vert_1$;
\item[(c)] $\sqrt{n} \, \zeta_2(F_n,G_{\hat\mu}) \cd \Vert \mathbb{X}_{\zeta_2,G_\mu} \Vert_1 = \mu^2\, \Vert \mathbb{X}_{\zeta_2,G_1} \Vert_1$;
\item[(d)] $\sqrt{n} \, \bar\zeta_2(F_n,G_{\hat\mu}) \cd \Vert \mathbb{X}_{\bar\zeta_2,G_1} \Vert_1 = \Vert \mathbb{X}_{\zeta_2,G_1} \Vert_1$.
\end{enumerate}
\end{corollary}

The following corollary states that when $X$ does not share any part of its distribution function with the exponential one, the limiting distribution $\delta_\infty(d,F)$ in (\ref{delta-infinity}) is actually normal, for all the considered distances.

\begin{corollary}\label{Corrolary normal}
Let us assume that the conditions of Theorem \ref{Theorem 3} hold and let us further assume that the set $I({g_d})$ (defined in Theorem \ref{Theorem 3}) has zero Lebesque measure. Then, $\delta_\infty(d,F)$ defined in (\ref{delta-infinity}) has a zero mean normal distribution.
\end{corollary}

The previous results could be useful to compute (asymptotic) confidence intervals for $d(F,G_\mu)$.
This can be implemented via the following bootstrap procedure: as $F$ is usually unknown, we substitute $F$ for $F_n$ in the limiting processes obtained in Theorems
\ref{Theorem SP Wasserstein} and \ref{Theorem SP Zolotarev} (equations (\ref{XFW})-(\ref{XFZB})).  Next, we simulate
a large number of trajectories of the processes to obtain a Monte Carlo approximation of the asymptotic distribution $\delta_\infty(d,F)$ given in Theorem \ref{Theorem 3}. Finally, we use the Monte Carlo sample quantiles to construct the desired interval. We also note that if we assume that the Lebesgue measure of the sets $I(g_d)$ is zero, then the procedure is simpler as Corollary \ref{Corrolary normal} ensures that the limit distribution is a zero mean normal distribution. Hence, in such a case it is enough to estimate the asymptotic variance of the limit via Monte Carlo and use the quantiles of a normal distribution. This latter interval is called \textit{standard normal interval} in \cite[p. 168]{Efron-Tibshirani}. However, let us observe that asymptotic confidence intervals could be unprecise when the sample size is small. As the distribution of $\delta_n(d,F)$ for a fixed $n$ could be extremely difficult to handle, an interesting alternative is to construct a bootstrap confidence interval of $d(F,G_\mu)$. In this situation, the \textit{percentile interval} proposed in \cite[p. 170]{Efron-Tibshirani} is a reasonable choice.

As $d(F,G_\mu)=0$ is equivalent to saying that $F$ is exponential, the considered distances can be additionally applied to goodness-of-fit tests for
$\text{H}_0: F \text{ is exponential}$.
As stated in Corollary \ref{Corollary exponential}, for $d=\bar\omega, \bar\zeta_2$, the asymptotic distribution of the test statistic $\sqrt{n} \, d(F_n,G_{\hat\mu})$ is completely determined under $\text{H}_0$. In practice, this result allows us to derive an asymptotic rejection region by Monte Carlo sampling from the asymptotic distribution. When the sample size $n$ is small, a better alternative is to use a parametric bootstrap procedure by sampling from an exponential distribution with mean $\hat \mu$.


\section{Simulations} \label{Section.Simulations}

The aim of this section is to analyze the finite-sample behaviour of the statistic $\delta_n(d,F)$ given in (\ref{normalized and estimated distance}), in particular to compare it with its asymptotic distribution $\delta_\infty(d,F)$ obtained in Theorem~\ref{Theorem 3}. We only consider the normalized versions of the distances, $d=\bar\omega$ and $d=\bar\zeta_2$.
In any case, the simulations results for the unnormalized metrics are similar and do not add relevant information.

To compute the normalized empirical distances $\bar\omega(F_n,G_{\hat\mu}) = \omega(F_n,G_{\hat\mu})/\hat\mu$ and $\bar\zeta_2(F_n,G_{\hat\mu}) = \zeta_2(F_n,G_{\hat\mu})/\hat\mu^2$, we have used the following equalities
\begin{equation} \label{EmpWassDist}
\omega(F_n,G_{\hat\mu}) = X_{(1)} - \hat\mu \, G_{\hat\mu}(X_{(1)}) + \int_{X_{(1)}}^{X_{(n)}} |F_n(x)-G_{\hat\mu}(x)|\,\dif x + \hat\mu \, e^{-X_{(n)}/\hat\mu}
\end{equation}
and
\begin{eqnarray}
\lefteqn{\zeta_2(F_n,G_{\hat\mu}) = 2 \int_0^\infty\left( \int_0^t (F_n(x)-G_{\hat\mu}(x))\,\dif x \right)_+ \, \dif t + \hat\mu^2 - \frac{a_2}{2}} \nonumber \\
 & & = 2 \int_{X_{(1)}}^{X_{(n)}} \left( -X_{(1)} + \hat\mu \, G_{\hat\mu}(X_{(1)}) + \int_{X_{(1)}}^t (F_n(x)-G_{\hat\mu}(x))\,\dif x \right)_+ \dif t + \hat\mu^2 - \frac{a_2}{2} , \label{EmpZolDist}
\end{eqnarray}
where $X_{(1)}\le\cdots\le X_{(n)}$ are the order statistics of the sample and $a_2 : = \sum_{i=1}^n X_i^2/n$. Even though the integrals appearing in (\ref{EmpWassDist}) and (\ref{EmpZolDist}) can be expressed in terms of the order statistics, from a computational viewpoint it is more convenient to approximate them numerically: this was carried out by discretizing the integral on the equispaced grid $X_{(1)}+ k\, \delta$ with $k=0,1,2,\ldots,20000$ and $\delta= (X_{(n)}-X_{(1)})/20000$.

The asymptotic distribution $\delta_\infty(d,F)$ can be approximately sampled by generating trajectories of the Brownian bridge $\B_{F}(t)$ on a bounded time interval $[0,T]$ and then approximating the integrals such as $\int_0^\infty \B_{F}(t) \, \dif t$ by their discretization on $[0,T]$. In our work, we have chosen $T$ as the smallest integer larger or equal to $F^{-1}(1-{\tt tol})$, where {\tt tol} is a tolerance limit equal to $10^{-6}$. The integral discretization was carried out with an equispaced grid on $[0,T]$ yielding 50000 subintervals.

Computing the normalized version, $d=\bar\omega$ or $d=\bar\zeta_2$, of the distances $d(F,G_\mu)$ is equivalent to computing the distance of the re-scaled variable $X/\mu$ to the exponential distribution with mean $\mu=1$. As a consequence, in this Monte Carlo study the data-generating distributions have expectation $\mu=1$. Specifically, we have considered:
\begin{list}{-}{\leftmargin=0.4cm \itemindent=0em}
\item the exponential distribution with mean $\mu=1$;
\item the Weibull distribution with shape parameter $a>0$ and scale parameter $\lambda = 1/\Gamma(1+1/a)$, with probability density
$f(x) = a/\lambda \, (x/\lambda)^{a-1} \, e^{-(x/\lambda)^a}$, for $x>0$;
\item the gamma distribution with shape parameter $a>0$ and scale parameter $\lambda=1/a$, with density
$f(x) = a^a/\Gamma(a) \, x^{a-1} \, e^{-ax}$, for $x>0$.
\end{list}
For the Weibull distribution with mean 1 we have chosen the values $a=0.9$ and $a=1.1$ for the shape parameter. For the gamma distribution with mean 1, we have used shape parameters $a=0.9$ and $a=1.1$ (see Figure~\ref{Figure.DensitiesSimulations}).

In the simulations we have generated 10000 samples of size $n$ =  100, 500, 1000 and 5000 from each of these distributions. For each Monte Carlo sample, we have computed the statistic $\delta_n(d,F)$ for $d=\bar\omega$ and $d=\bar\zeta_2$. The programming language used in this work is R (\url{www.R-project.org}). The results of the simulations are summarized in Figures~\ref{Figure.Boxplots.Exp}, \ref{Figure.Boxplots.Wei11}, \ref{Figure.Boxplots.Wei09}, \ref{Figure.Boxplots.Gam11} and \ref{Figure.Boxplots.Gam09}. Each figure shows the evolution, as $n$ increases, of the boxplots of these statistics $\delta_n(d,F)$ towards the boxplot of its asymptotic distribution $\delta_\infty(d,F)$ given in (\ref{delta-infinity}). This latter boxplot is also based on 10000 samples from the corresponding limit distribution.

On the one hand, we observe that, the finite-sample behavior of $\delta_n(d,F)$, both for $d=\bar\omega$ and for $d=\bar\zeta$, is stable for the exponential distribution. In particular, the quartiles and median of $\delta_n(d,F)$ are very similar for any $n$. For nonexponential distributions $F$, the boxplot of $\delta_n(d,F)$ resembles that of $\delta_\infty(d,F)$ for large sample sizes ($n\geq 1000$). 
On the other hand, we also observe that, the closer $F$ is to the exponential distribution, the larger $n$ has to be for the distribution of $\delta_n(d,F)$ to approach its limit.
This is reasonable: when $F$ is not exponential, but close to it and the set $I({g_d})$ has zero Lebesgue measure, the finite sample behavior of $\delta_n(d,F)$ is almost as if $F$ were exponential, but the limit is actually Gaussian (see Corollary~\ref{Corrolary normal}).

\begin{figure}
\begin{center}
\includegraphics[width=0.6\textwidth]{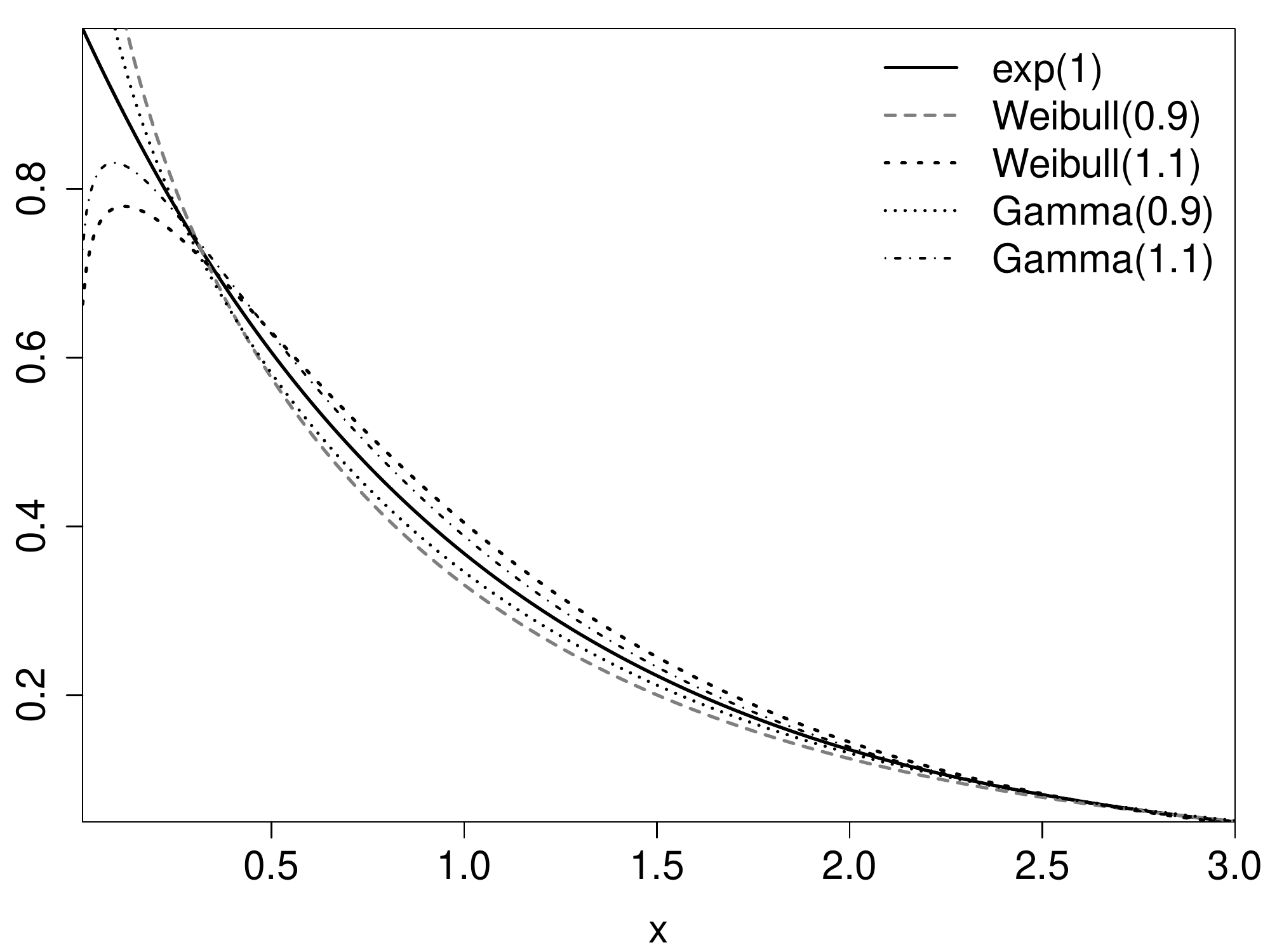}
\end{center}
\caption{Probability densities in the simulation study of Section~\ref{Section.Simulations}.}
\label{Figure.DensitiesSimulations}
\end{figure}

\begin{figure}
\begin{center}
\begin{tabular}[b]{cc}
    \includegraphics[width=0.45\textwidth]{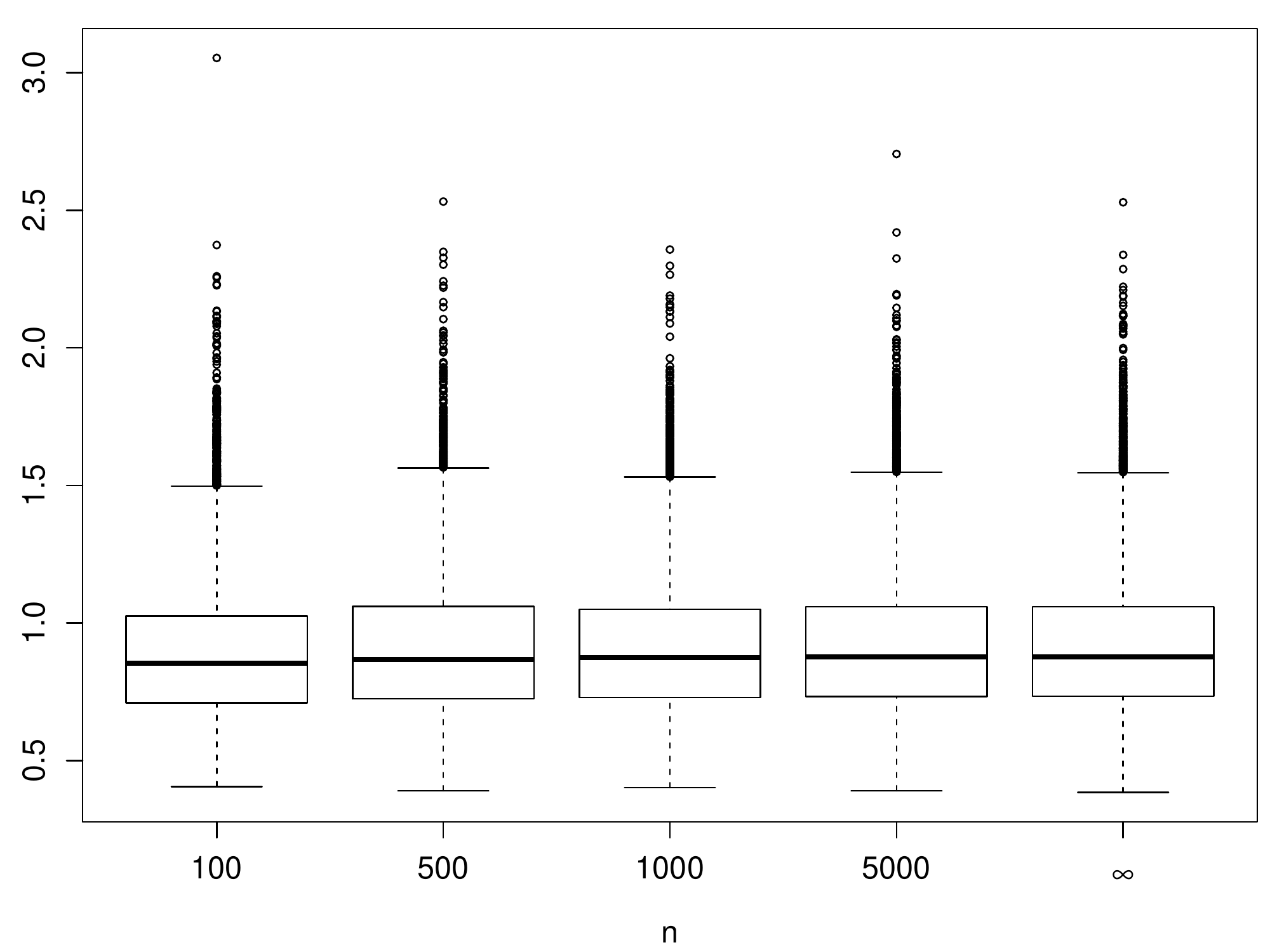} &
    \includegraphics[width=0.45\textwidth]{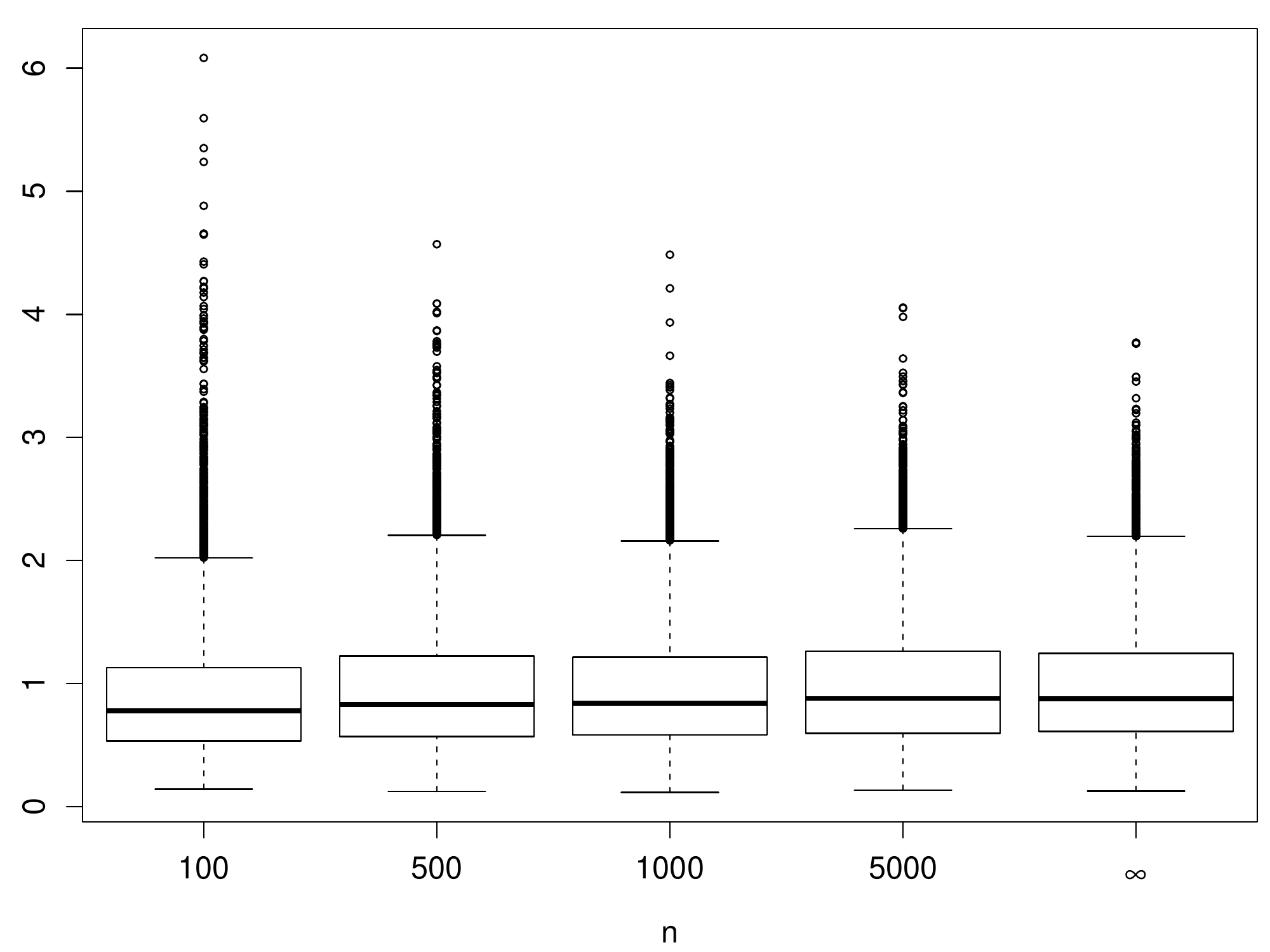} \\
    (a) & (b)
    \end{tabular}
\end{center}
\caption{Boxplots of $\delta_n(d,F)$ and its asymptotic distribution $\delta_\infty(d,F)$ for (a) $d=\bar\omega$ and (b) $d=\bar\zeta_2$. The distribution $F$ is exponential(1).}
\label{Figure.Boxplots.Exp}
\end{figure}

\begin{figure}
\begin{center}
\begin{tabular}[b]{cc}
    \includegraphics[width=0.45\textwidth]{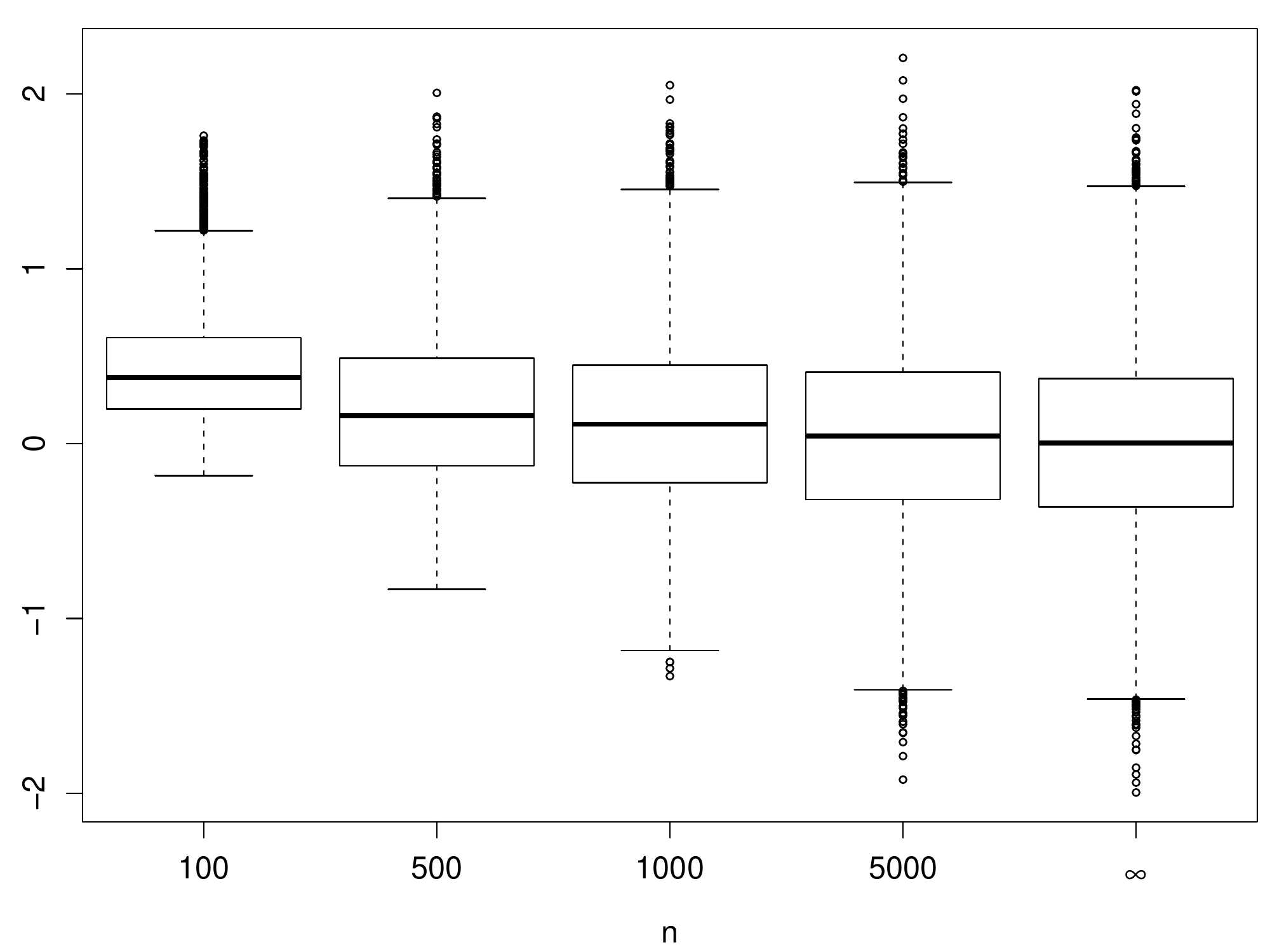} &
    \includegraphics[width=0.45\textwidth]{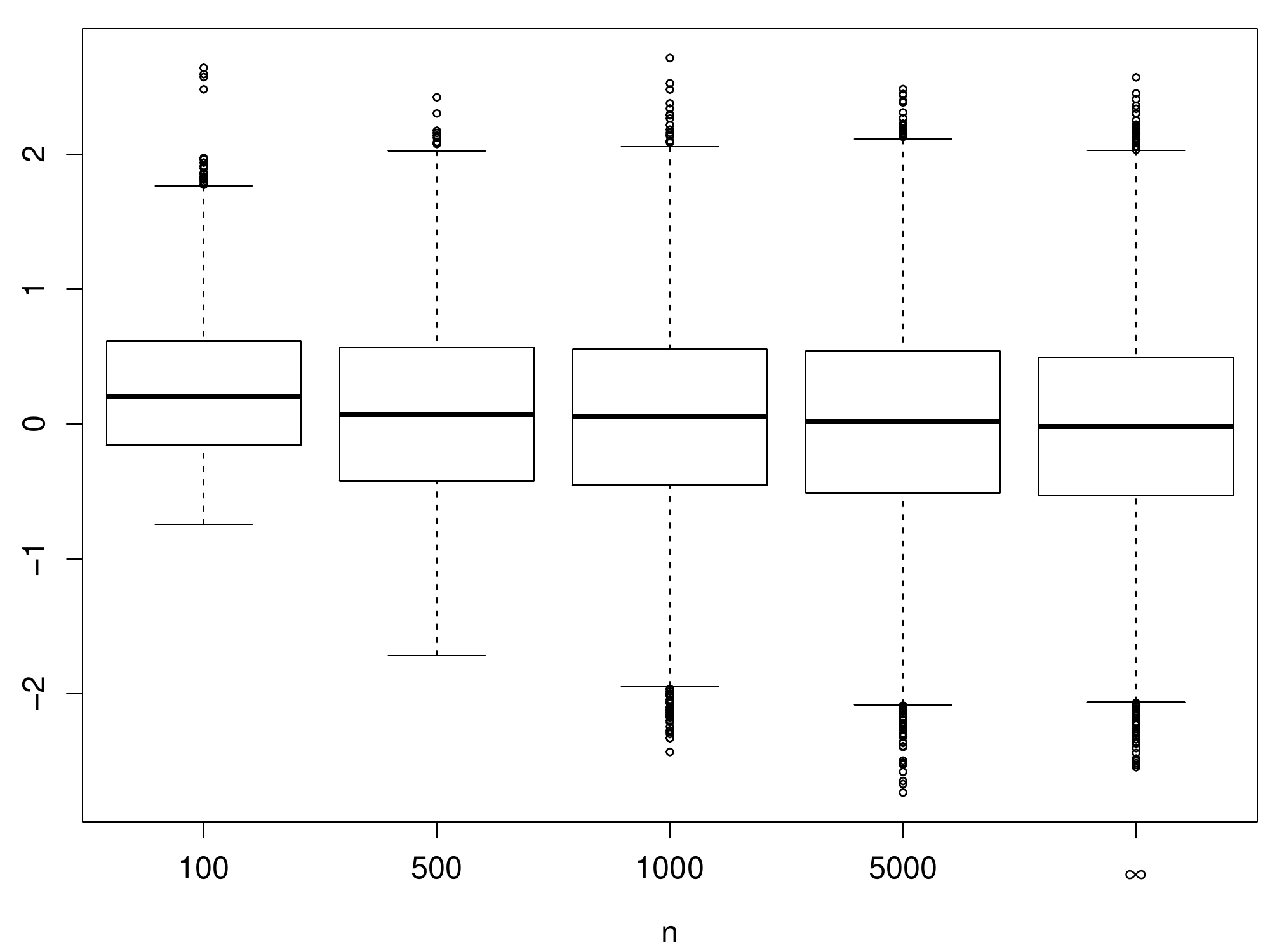} \\
    (a) & (b)
    \end{tabular}
\end{center}
\caption{Boxplots of $\delta_n(d,F)$ and its asymptotic distribution  $\delta_\infty(d,F)$ for (a) $d=\bar\omega$ and (b) $d=\bar\zeta_2$. The distribution $F$ is Weibull with shape parameter $a=1.1$.}
\label{Figure.Boxplots.Wei11}
\end{figure}

\begin{figure}
\begin{center}
\begin{tabular}[b]{cc}
    \includegraphics[width=0.45\textwidth]{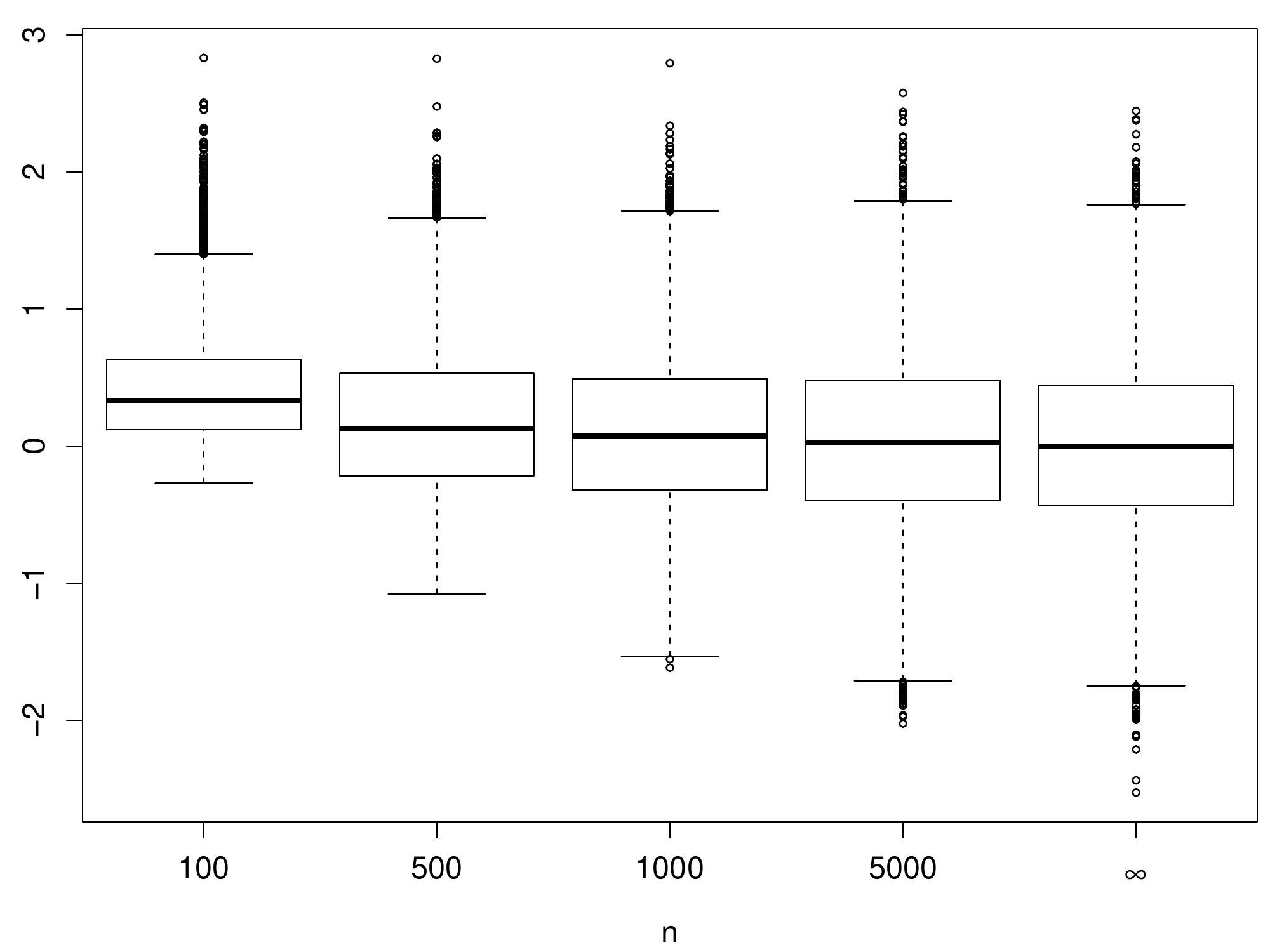} &
    \includegraphics[width=0.45\textwidth]{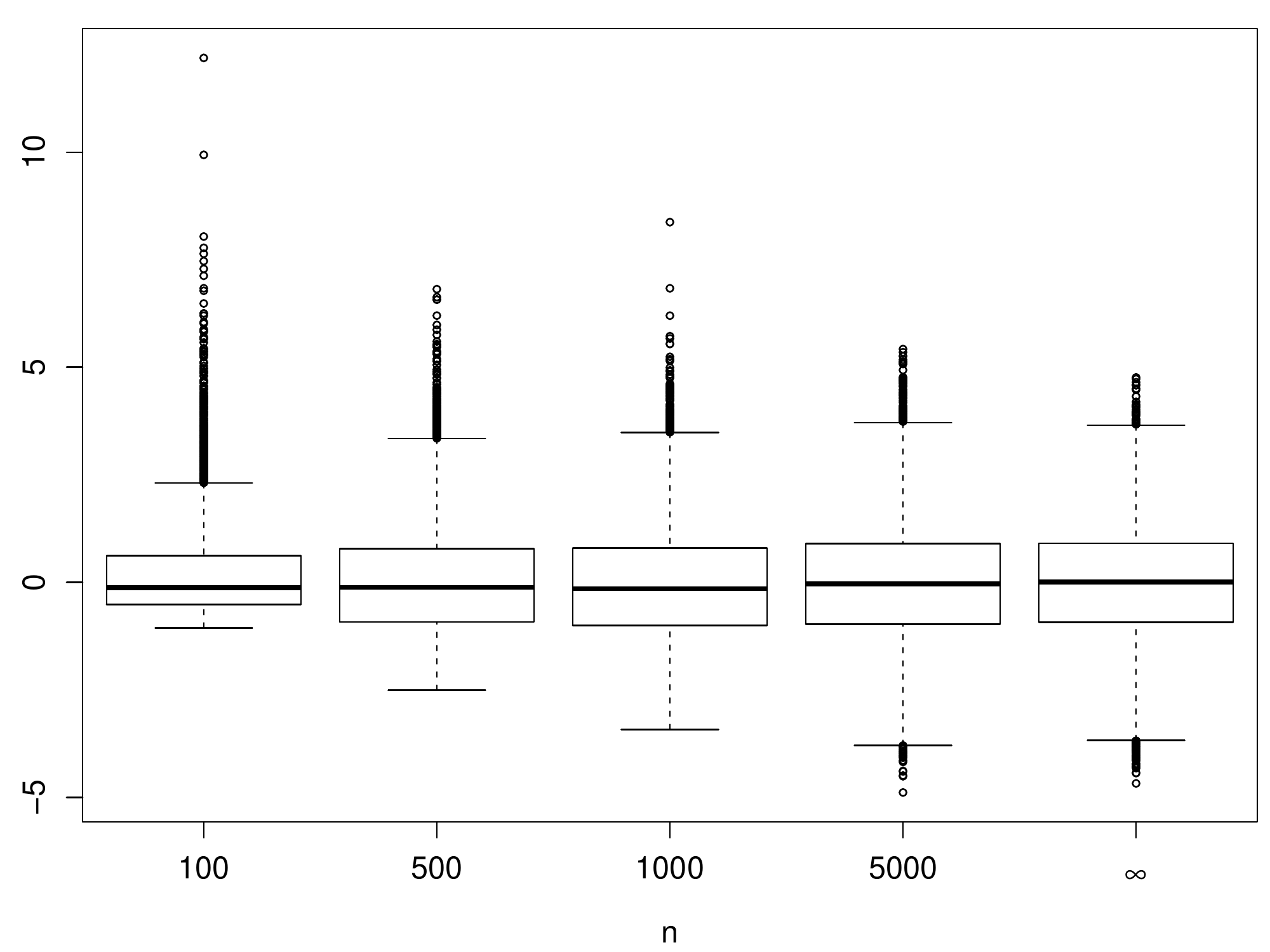} \\
    (a) & (b)
    \end{tabular}
\end{center}
\caption{Boxplots of $\delta_n(d,F)$ and its asymptotic distribution  $\delta_\infty(d,F)$ for (a) $d=\bar\omega$ and (b) $d=\bar\zeta_2$. The distribution $F$ is Weibull with shape parameter $a=0.9$.}
\label{Figure.Boxplots.Wei09}
\end{figure}

\begin{figure}
\begin{center}
\begin{tabular}[b]{cc}
    \includegraphics[width=0.45\textwidth]{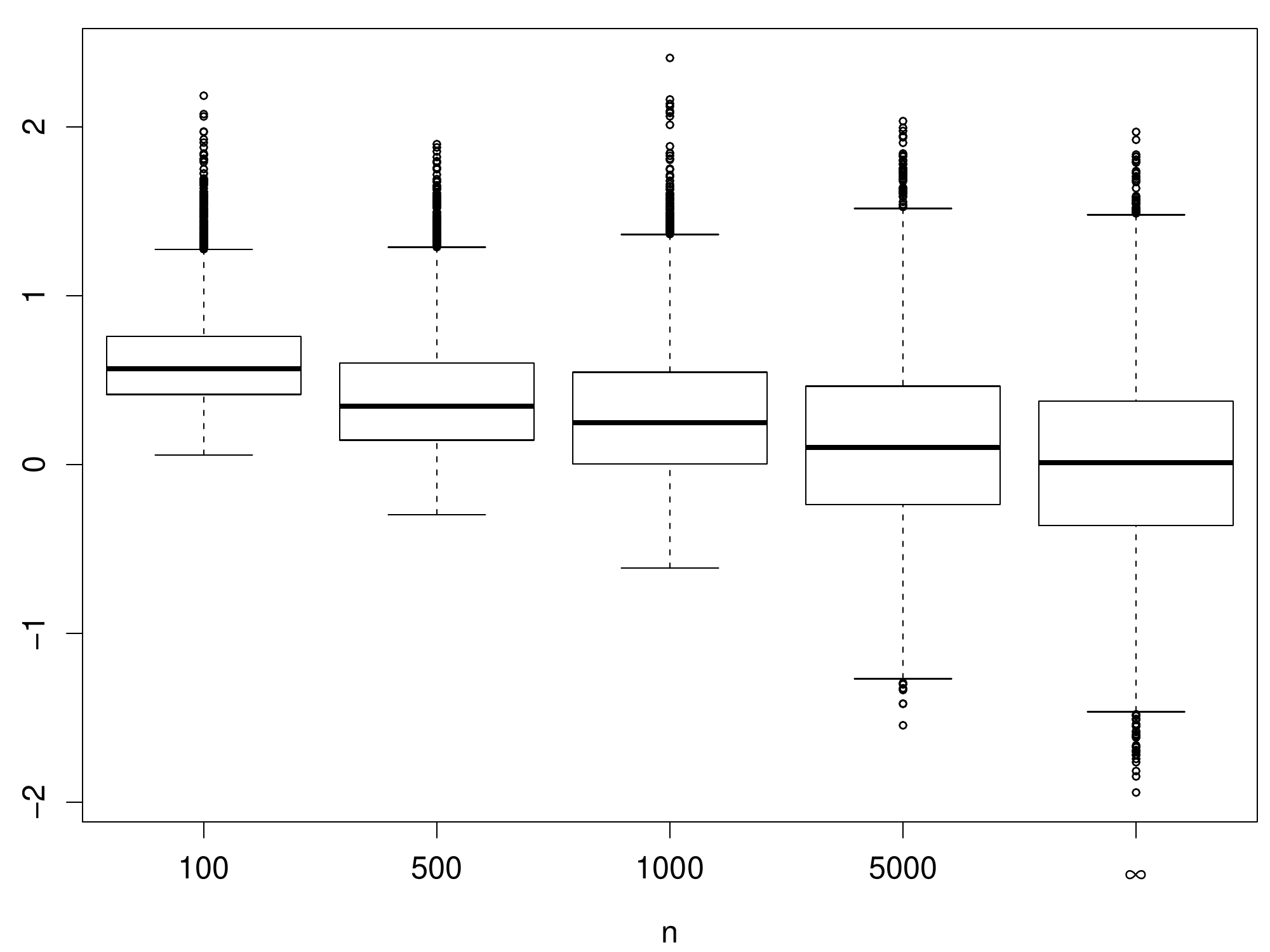} &
    \includegraphics[width=0.45\textwidth]{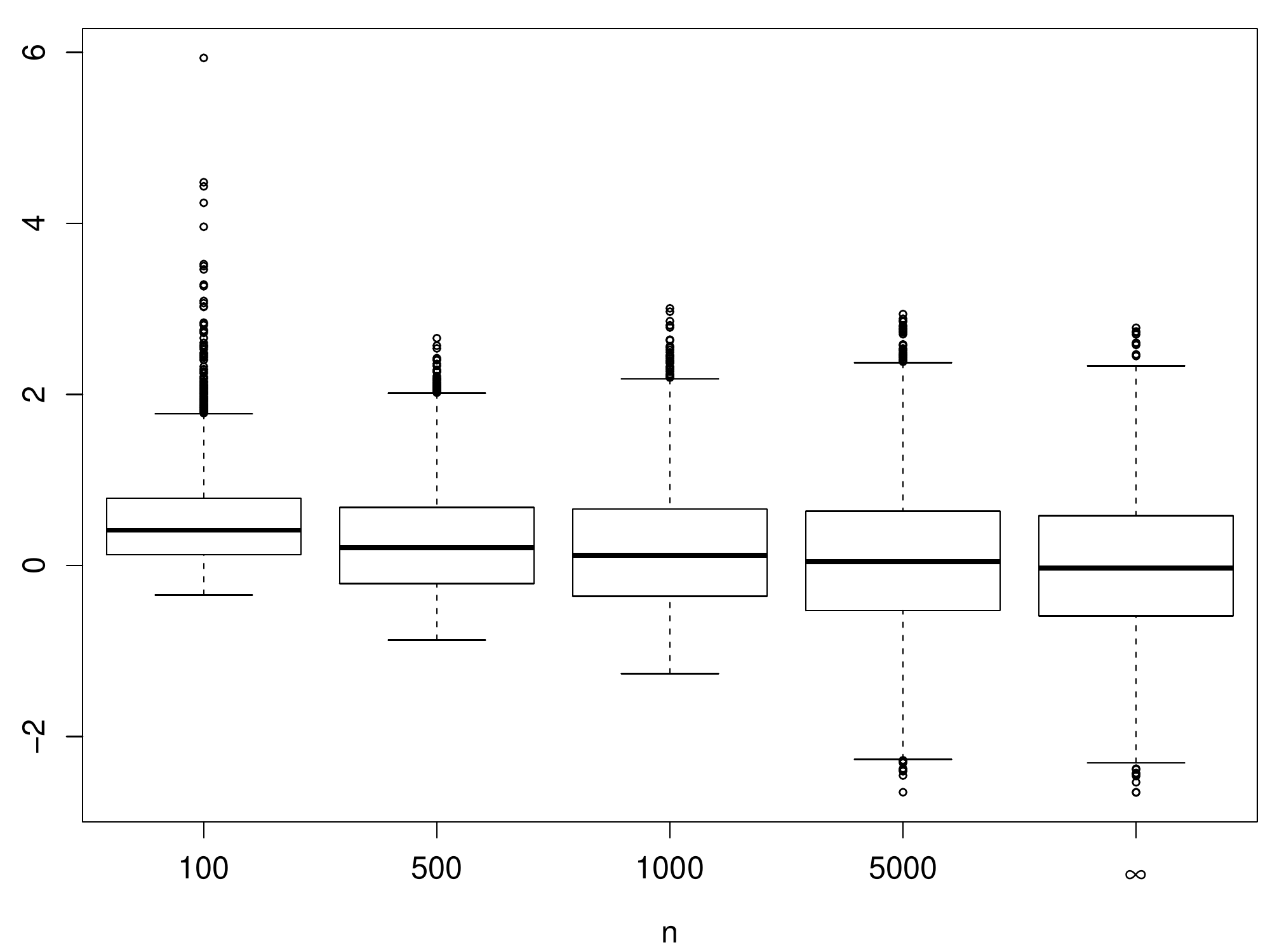} \\
    (a) & (b)
    \end{tabular}
\end{center}
\caption{Boxplots of $\delta_n(d,F)$ and its asymptotic distribution  $\delta_\infty(d,F)$ for (a) $d=\bar\omega$ and (b) $d=\bar\zeta_2$. The distribution $F$ is gamma with shape parameter $a=1.1$.}
\label{Figure.Boxplots.Gam11}
\end{figure}

\begin{figure}
\begin{center}
\begin{tabular}[b]{cc}
    \includegraphics[width=0.45\textwidth]{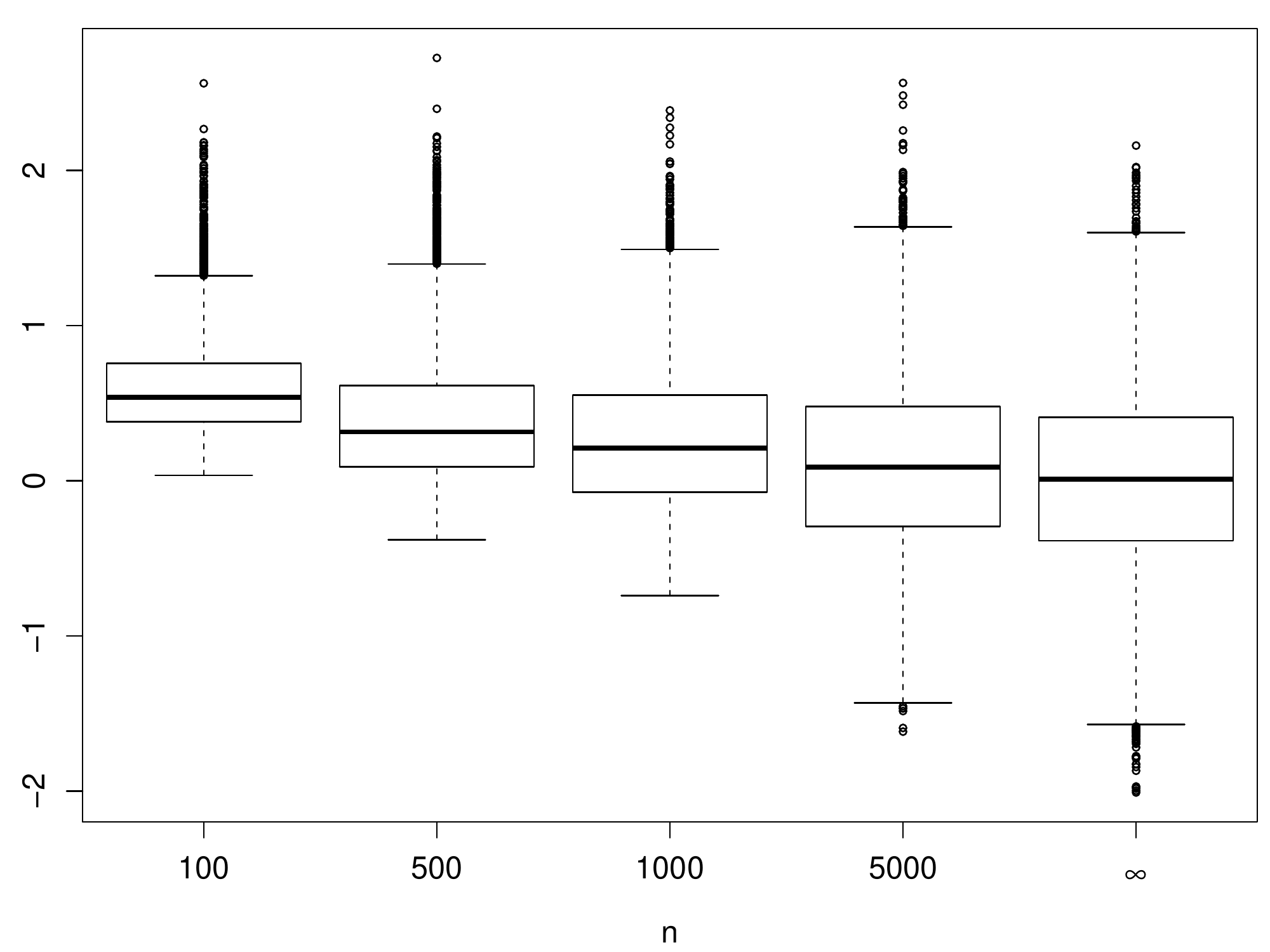} &
    \includegraphics[width=0.45\textwidth]{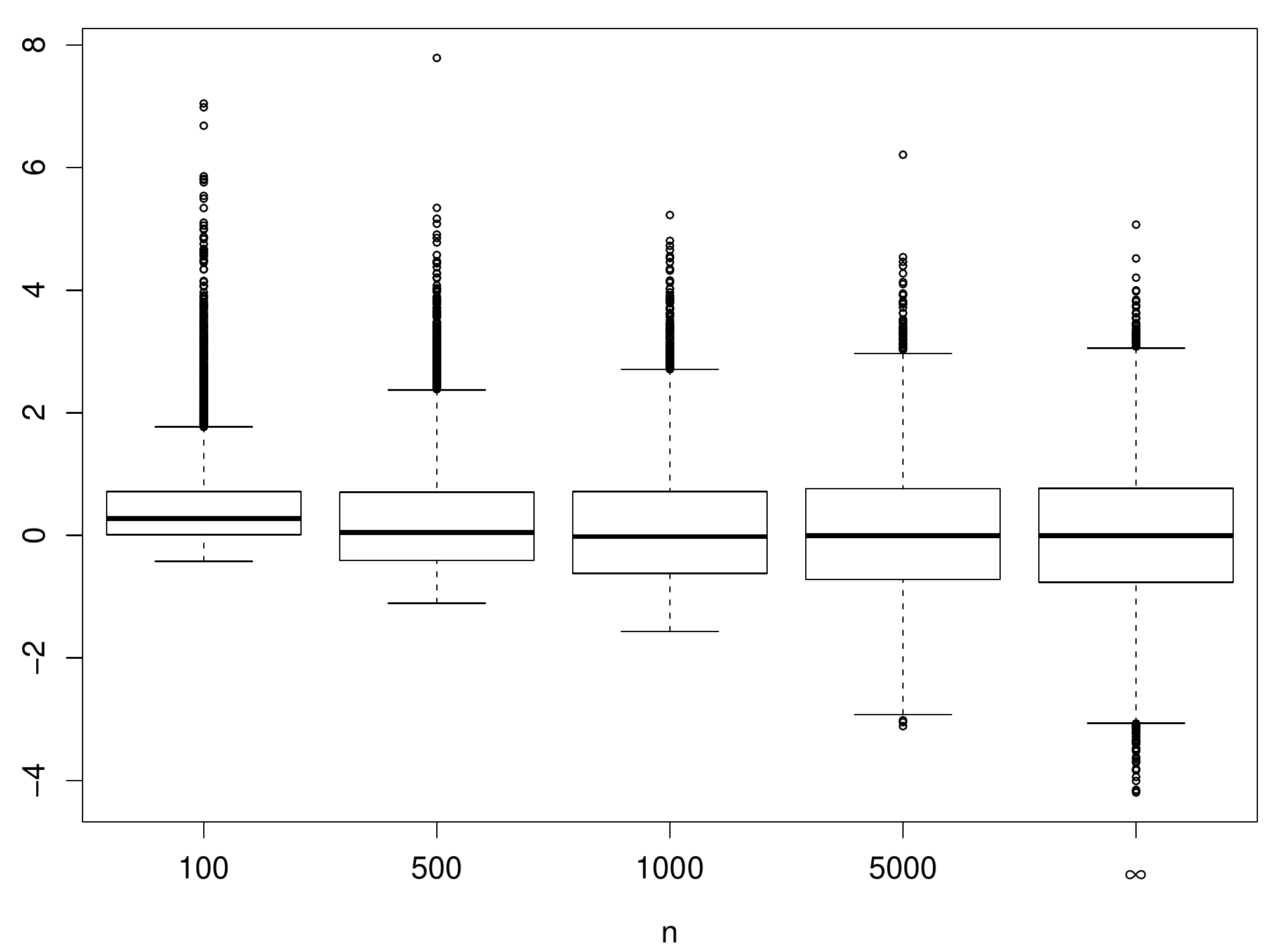} \\
    (a) & (b)
    \end{tabular}
\end{center}
\caption{Boxplots of $\delta_n(d,F)$ and its asymptotic distribution  $\delta_\infty(d,F)$ for (a) $d=\bar\omega$ and (b) $d=\bar\zeta_2$. The distribution $F$ is gamma with shape parameter $a=0.9$.}
\label{Figure.Boxplots.Gam09}
\end{figure}


\section{Analysis of the COUP data} \label{Section COUPAnalysis}

As mentioned in the introduction, photon interarrival times resulting from extragalactic radiation are usually well-modeled by exponential distributions, while, for the classes of PMS stars, PIT distributions might deviate from the exponential one. Hence, it is natural to compare different sources by computing the distances between their corresponding PIT and an exponential variable with the same mean.
In this section, we compute the empirical normalized Wasserstein and Zolotarev distances, $\bar\omega(F_n,G_{\hat\mu})$ and $\bar\zeta_2(F_n,G_{\hat\mu})$, to the exponential class for the sample of PIT of each X-ray source in the COUP dataset, described in Sections~\ref{Section_Introduction} and~\ref{Section_COUP_Dataset}. For the sake of comparison, we also include the usual Kolmogorov metric $\kappa(F_n,G_{\hat\mu})$ in the analysis as it is commonly used by astrophysicists in the classification of X-ray sources. The three types of sources of interest (namely, lightly-obscured PMS stars, heavily-obscured PMS stars and extragalactic) are compared via the values of these metrics.

For each of the X-ray sources in the three groups, we have computed the series of times between consecutive photon detections, taking into account the five observation gaps due to the passages through the high-radiation belts (see Section~\ref{Section_COUP_Dataset}).
We have also kept in mind that the complicated Chandra ACIS X-ray background, which is the combination of both celestial and instrumental backgrounds, tends to have peaks at energies below 0.5 keV and above 8 keV. To remove a significant portion of the X-ray background, we have followed \cite{Getman_etal05a}, who apply an energy filter to the data by cleaning X-ray events with energies out of the total (0.5-8) keV band.

In order to obtain reasonably good estimates of the distances to the exponential distribution, we have kept only the COUP series with at least 100 PIT. We have finally analyzed 1090 samples of PIT, of which 73, 644 and 373 had been previously classified as extragalactic sources, lightly-obscured and heavily-obscured PMS stars, respectively.
For each of these 1090 COUP sources, we have computed the distances $\kappa(F_n,G_{\hat\mu})$, $\bar\omega(F_n,G_{\hat\mu})$ and $\bar\zeta_2(F_n,G_{\hat\mu})$ of the empirical distribution of PIT to the exponential distribution with the same sample mean.

The interesting issue is whether the distance of PIT to the exponential distribution depends on the source class. The results are summarized in Figures~\ref{Figure.BoxplotsCOUP} and~\ref{Figure.DistribFnsCOUP}. In Figure~\ref{Figure.BoxplotsCOUP} we have displayed the boxplots of $\log(d(F_n,G_{\hat\mu}))$, for $d=\kappa,\, \bar\omega,\, \bar\zeta_2$, separated according to the three types of COUP sources. Outliers have been identified by the COUP source number (see \cite{Getman_etal05a}). As expected, for the three distances we see that the distribution of PIT due to extragalactic radiation is the nearest to the exponential distribution.
Moreover, Figure~\ref{Figure.DistribFnsCOUP} displays the empirical distribution functions of $\sqrt{n} \, d(F_n,G_{\hat\mu})$ ($d = \bar\omega,\, \bar\zeta_2$), for the different COUP groups and the distribution function of $\delta_\infty(d,G_1)$ when the distribution function is exponential with mean 1. The empirical distribution function corresponding to the extragalactic group is again the nearest to the distribution function of $\delta_\infty(d,G_1)$, the asymptotic distribution in Corollary~\ref{Corollary exponential}. Additionally, Figure~\ref{Figure.BoxplotsCOUP} shows that the normalized Zolotarev distance separates better the three classes than the normalized Wasserstein and the Kolmogorov-Smirnov metrics.
Observe also that the PIT corresponding to lightly-obscured PMS stars are noticeably nearer to the exponential than those of the heavily-obscured group.
A plausible explanation for this difference is the following: X-ray emission in heavily-obscured PMS stars is more affected by the intervening interstellar absorption, mainly the absorption from the gas in the molecular cloud and/or local absorption in an envelope or a disk around a star. Other factor might also be that the heavily absorbed sample is generally younger and more ``diskier'', in the sense that it has a higher fraction of stars that are still surrounded by their circumstellar disks; some of them highly accreting. Thus, the age and the presence/absence of disks could play an additional role, by affecting the X-ray production mechanisms of the strong X-ray flares in PMS stars and in turn affecting their X-ray photon arrival patterns (see \cite{Getman_etal08a} and \cite{Getman_etal08b} for differences in the strong X-ray flares between the diskless and disky-accreting stellar populations).

\begin{figure}
\begin{center}
\begin{tabular}[b]{c}
    \includegraphics[width=0.55\textwidth]{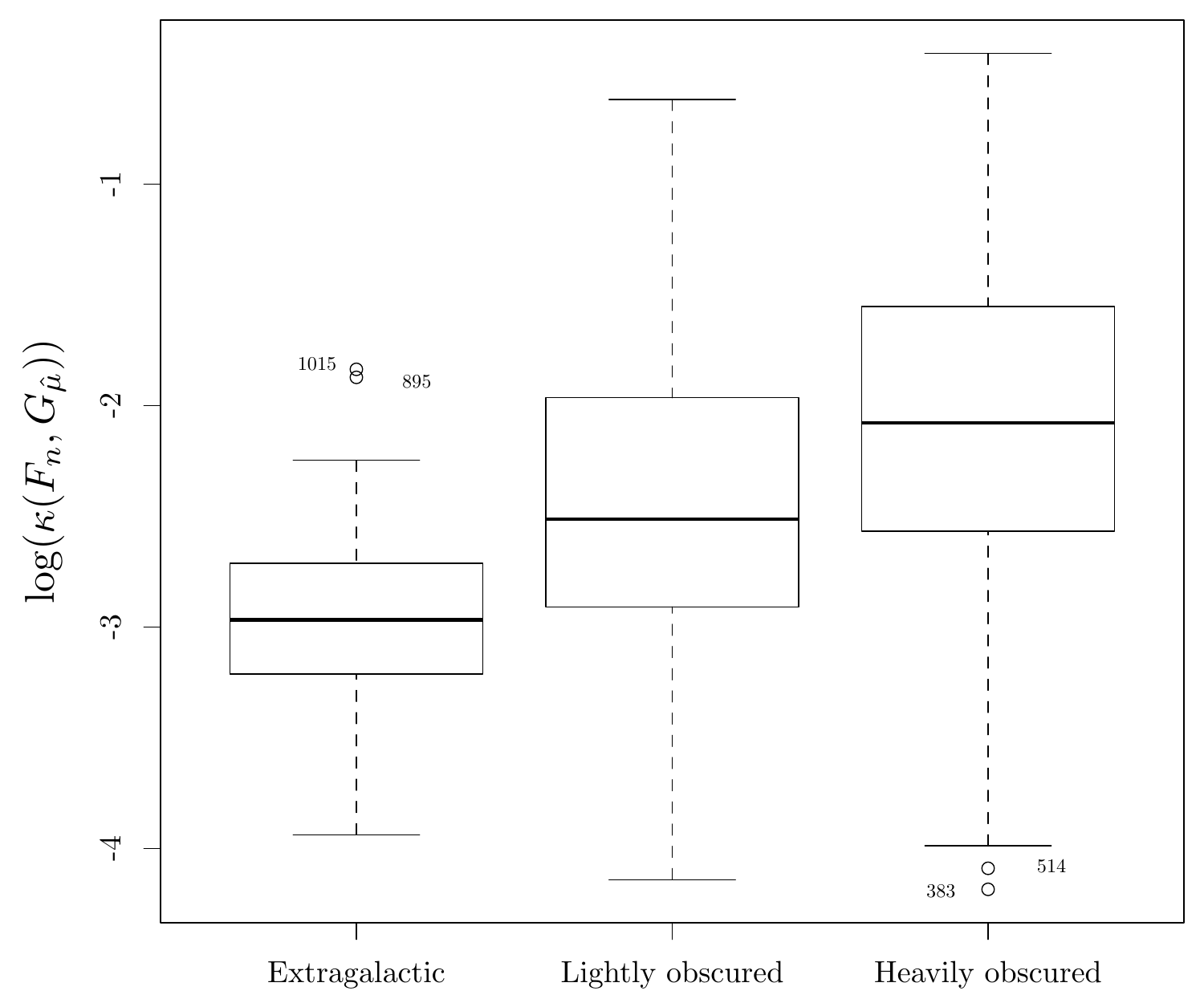} \\
    (a) \\ [2 mm]
    \includegraphics[width=0.55\textwidth]{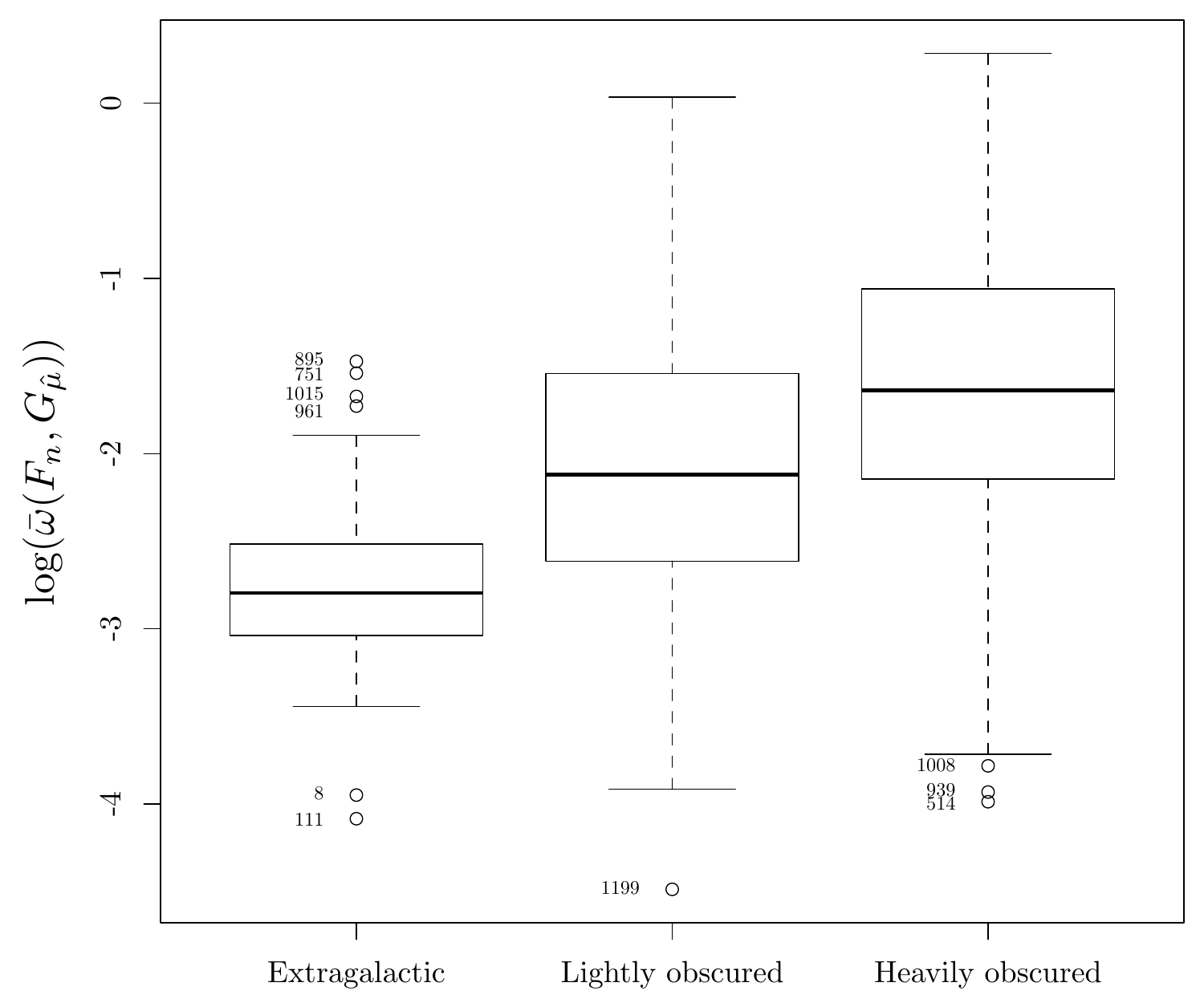} \\
    (b) \\ [2 mm]
    \includegraphics[width=0.55\textwidth]{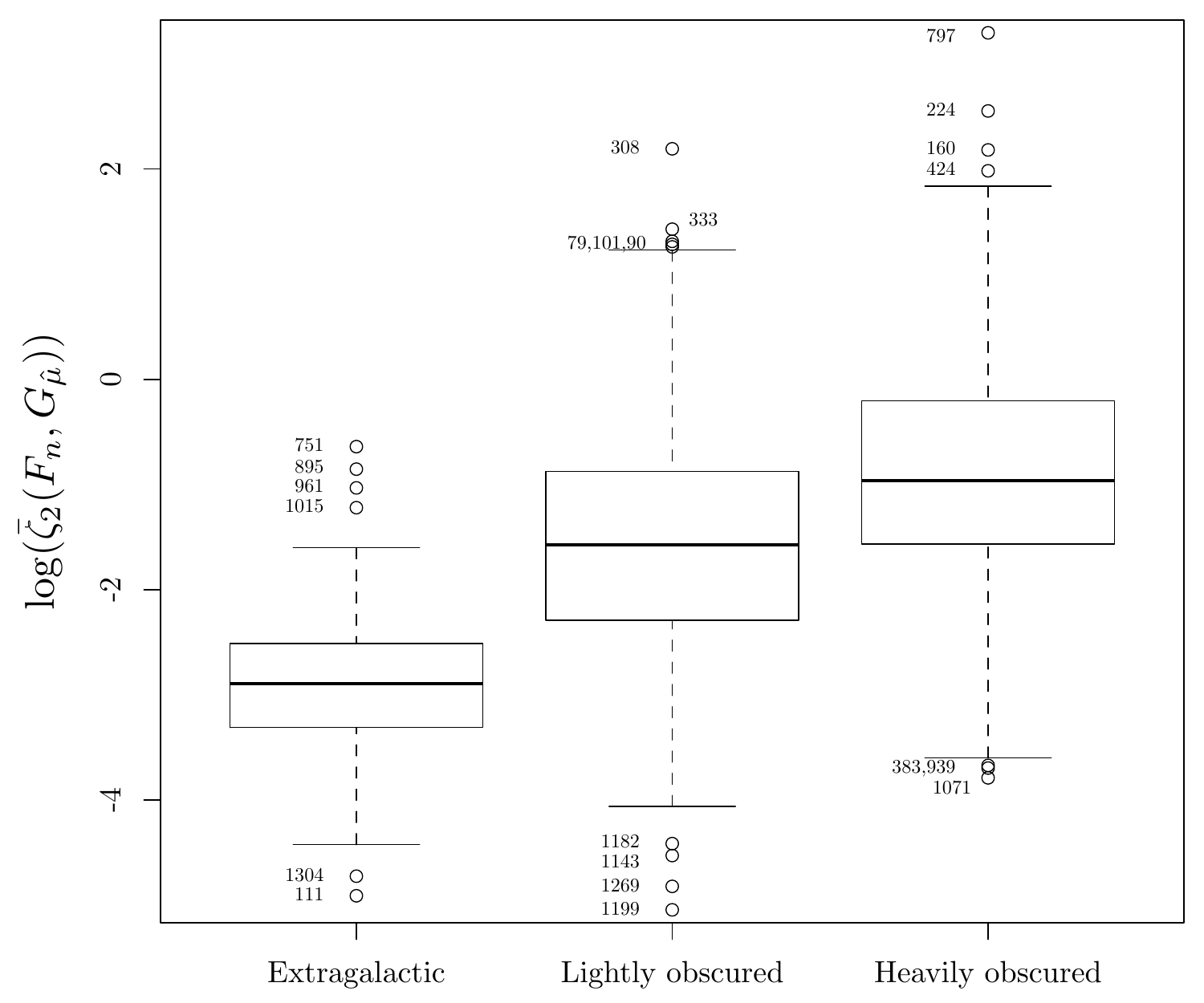} \\
    (c) \\
    \end{tabular}
\end{center}
\caption{Analysis of COUP data. Boxplots of $\log(d(F_n,G_{\hat\mu}))$ for (a) $d=\kappa$, (b) $d=\bar\omega$ and (c) $d=\bar\zeta_2$.}
\label{Figure.BoxplotsCOUP}
\end{figure}

\begin{figure}
\begin{center}
\begin{tabular}[b]{cc}
    \includegraphics[width=0.45\textwidth]{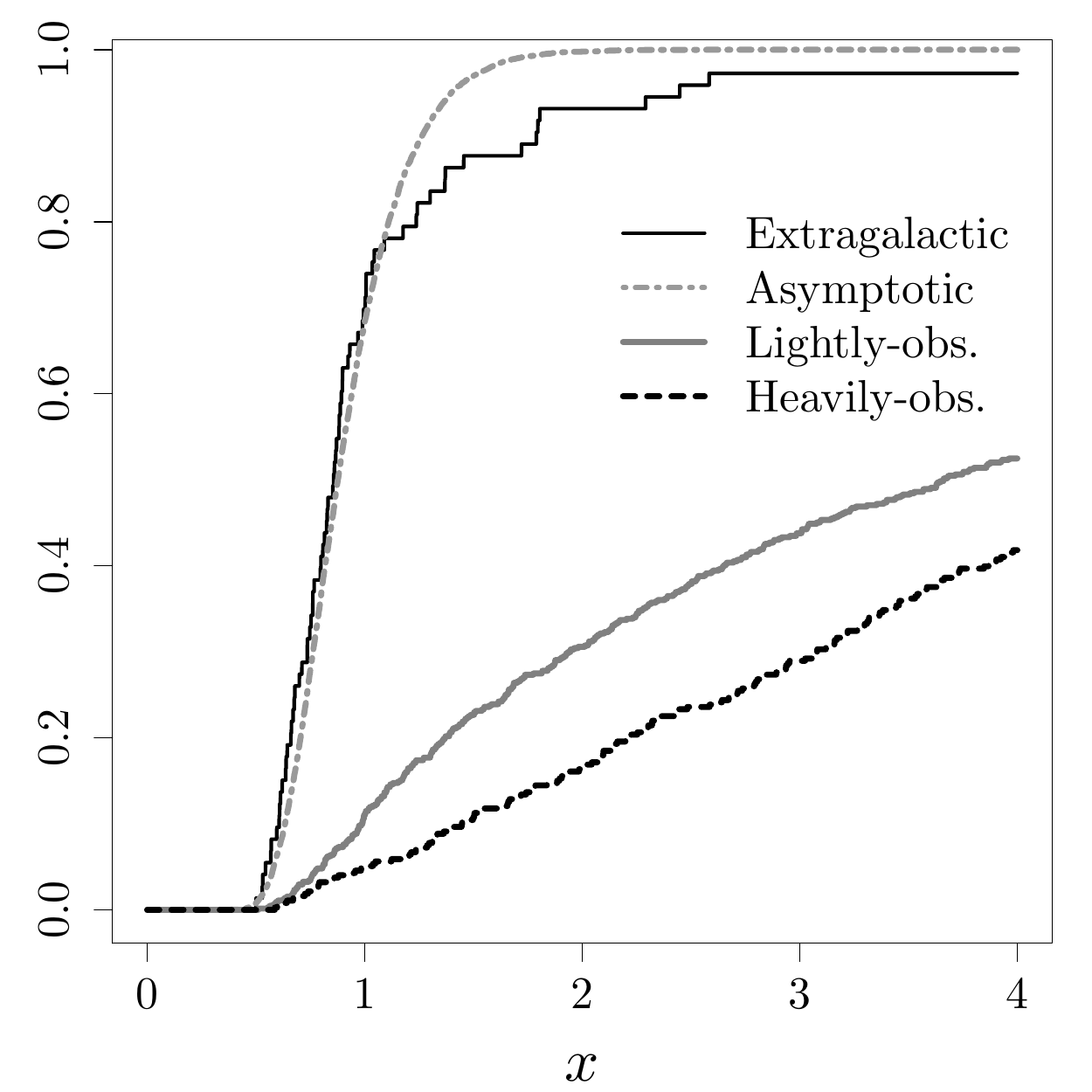} &
    \includegraphics[width=0.45\textwidth]{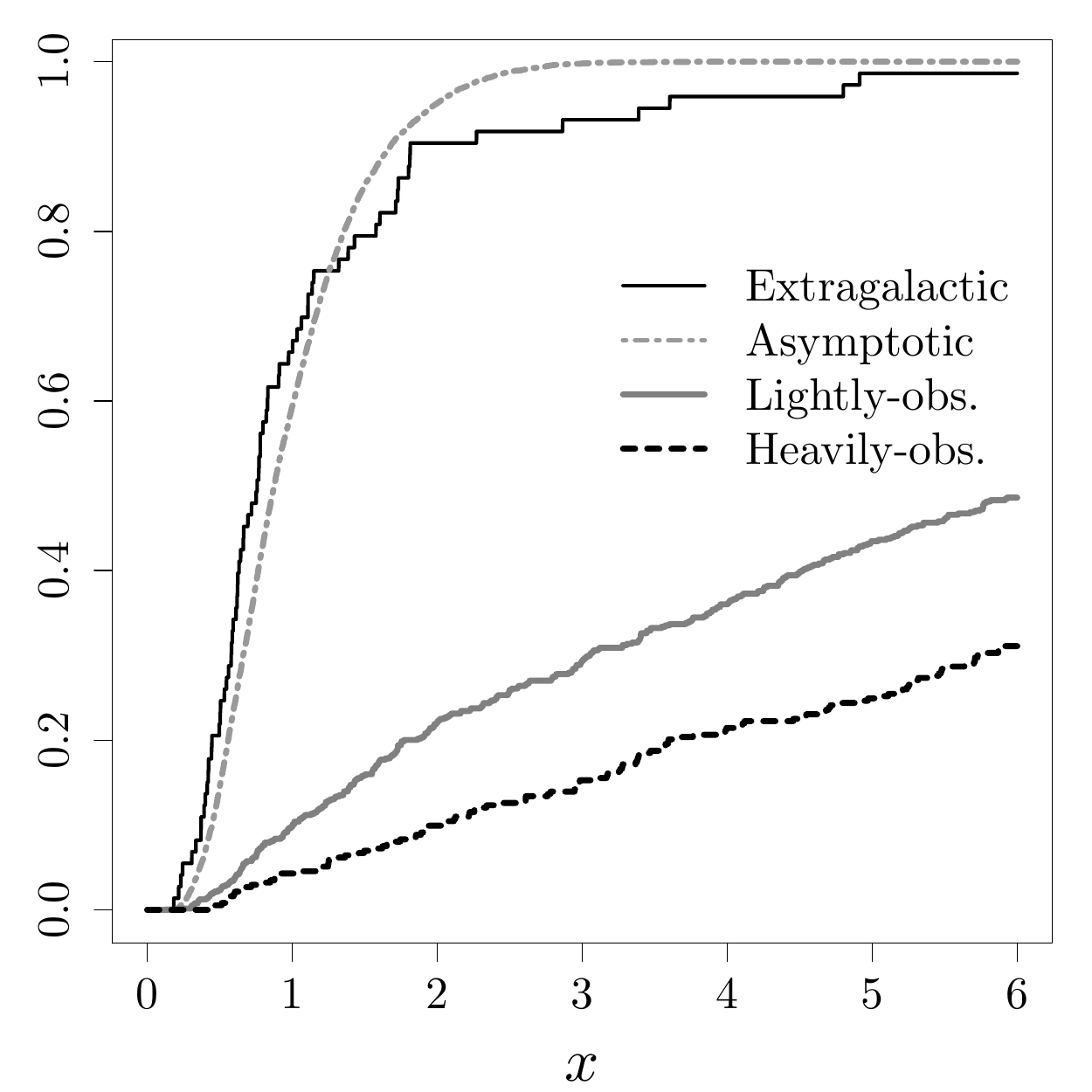} \\
    (a) & (b)
    \end{tabular}
\end{center}
\caption{Analysis of COUP data. Distribution function of $\delta_\infty(d,G_1)$ (as given in Corollary 1) and empirical distribution functions of $\sqrt{n} \, d(F_n,G_{\hat\mu})$ for (a) $d=\bar\omega$ and (b) $d=\bar\zeta_2$.}
\label{Figure.DistribFnsCOUP}
\end{figure}

The normalized metrics $\bar\omega$ and $\bar\zeta_2$ are also capable of detecting outlying COUP sources (see Figure~\ref{Figure.BoxplotsCOUP}), which are interesting to interpret.
The X-ray light-curves and the photon arrival diagrams in the COUP atlas (see \cite[Figure Set 12]{Getman_etal05a}), a numerical and graphical summary of all the COUP sources, show one thing in common among different types of outliers: the outlying sources nearest to the exponential distribution do not exhibit long and powerful X-ray flares, while all outliers farthest from 0 do.
In other words, some outliers in Figure~\ref{Figure.BoxplotsCOUP} could be in fact misclassified X-ray sources. For instance, the COUP sources 751, 895, 961, and 1015, classified as extragalactic sources by \cite{Getman_etal05b} and having the highest normalized Wasserstein and Zolotarev distances to the exponential distribution among those sources of the extragalactic sample (Figure \ref{Figure.BoxplotsCOUP}), are likely to be young stellar objects in the Orion Nebula region.

To analyze the nature of the above-mentioned four outliers with more precision, we decided to incorporate another informative variable with the potential of discriminating among the three COUP classes specified above. Taking into account that, for each of the photons belonging to a COUP source, we observe both the photon arrival time and its associated energy, it is natural to choose some feature summarizing the energies in the COUP series.
The COUP atlas provides extra information (in the form of tabulated quantities for each COUP source). Specifically, we have considered the quantity MedEn, the median energy (in keV) of the source, due to its strong correlation with the absorbing column density characterizing interstellar absorption (see \cite{Feigelson_etal05}). To further clarify why MedEn is an adequate choice, in Figure~\ref{Figure.COUPEnergiesDistribFns} we have plotted the empirical distribution functions of the photon energies rescaled to [0,1] for all the COUP sources. We clearly see that MedEn discriminates well between lightly obscured stars and extragalactic sources whereas distances to the exponential class separate better extragalactic sources from heavily obscured ones. Consequently, together these two quantities  might separate well the three considered COUP classes, though, on their own, each of these variables fails to achieve a low misclassification rate. To check this, we have carried out three classification procedures on these data: quadratic discriminant analysis, the $k$-nearest neighbours ($k$-NN) rule and model-based discriminant analysis, respectively implemented in the R packages MASS (\cite{Venables-Ripley}), class (\cite{Venables-Ripley}) and mclust (\cite{mclust_ref}). We can choose $k$ via cross-validation (CV). Using the package caret (\cite{Kuhn}), we have run 10-, 5-fold and leave-one-out (LOO) CV to get insight into the value of $5\leq k\leq 63$ yielding the largest accuracy. In the case of Zolotarev distance $\bar\zeta_2$ the optimal value of $k$ is 5 for the three CV procedures. For the Wasserstein metric $\bar\omega$, the optimal $k$'s are 35, 33 and 39 for 10-, 5-fold and LOO CV, respectively. For the Kolmogorov metric $\kappa$, the optimal $k$'s are 19, 17 and 21 (or 23 and 25) for 10-, 5-fold and LOO CV, respectively. Since, for each metric, the optimal values of $k$ are similar, we have chosen $k=5$ for $k$-NN with the $\bar\zeta_2$ metric, $k=35$ with $\bar\omega$ and $k=19$ in the case of $\kappa$.
The percentage of correct classifications with the three discriminant methods and the three distances appear in Table~\ref{TableCOUPClassification}: they are all nearly the same and remarkably high (around $90\%$). The normalized Zolotarev distance of the PIT to the exponential distribution achieves the best correct classification rate, no matter which procedure is used. Although the improvement over the usual Kolmogorov metric may not seem significant, this is probably due to the low proportion of extragalactic sources in the sample. To gain a better insight of the advantages of the Zolotarev metric, in  Table \ref{TableCOUPConfusionMatrices} we display the confusion matrices of these classification procedures.
If we look at the extragalactic class (NM), using the Kolmogorov distance provides positive predictive values (PPVs) of 67\%, 65\% and 68\%, for the quadratic, the $k$-NN and the mixture classification rules respectively. With the Zolotarev metric the corresponding PPVs are 78\%, 76\% and 78\%, respectively. This significant increase in the classification accuracy was one of the motivations of this work.

There could be concerns regarding the potential influence of the outlying sources and their masking effect on the procedures employed in this section. For example, classical discriminant methods (like the quadratic or model-based rules) are sensitive towards outliers and may result in inaccurate parameter estimation or ill-posed problems. These methods can be replaced by robustified classifiers that rely, e.g., on depth functions (\cite{Hubert-Rousseeuw-Segaert}) or on robust estimates of the mixture parameters (\cite{Garcia-Escudero}).
However, robust versions of the normalized distances of the PIT distributions to the exponential class are not adequate for this problem. The reason is that outlying interarrival times are precisely the ones leading to detection of flares, inconsistent with the usual behaviour of extragalactic radiation.


\begin{table}
\caption{Analysis of COUP data: Percentage of correct classifications based on the logarithm of the median photon energy joint with the logarithm of the Kolmogorov distance (first row), the normalized Wasserstein distance (second row) and the normalized Zolotarev distance to the exponential distribution (last row).}
\label{TableCOUPClassification}
\begin{center}
\begin{tabular}{lccc} 
Distance & Quadratic & $k$-NN & Model-based \\ \hline
Kolmogorov & 88.99 & 87.89 & 89.08 \\
Wasserstein & 89.54 & 87.98 & 88.99 \\
Zolotarev & 90.18 & 90.09 & 90.64 \\ \hline
\end{tabular}
\end{center}
\end{table}

{\begin{table}
\caption{Analysis of COUP data: confusion matrices for classifications based on the logarithm of the median photon energy and the logarithm of the Kolmogorov distance (first row), the normalized Wasserstein distance (second row) and the normalized Zolotarev distance (last row) to the exponential distribution. Numbers in bold font correspond to the highest correct classification in each class. NM = non members, HO = Heavily obscured, LO = Lightly obscured stars.}
\label{TableCOUPConfusionMatrices}
\small
\begin{center}
\begin{tabular}{@{\hspace{0mm}}c@{\hspace{1mm}}ccc}
& \multicolumn{3}{c}{\bf Classification rule} \\ [2 mm]
\bf Metric & \bf Quadratic & \bf $k$-NN & \bf Mixture \\ [1 mm]
\rotatebox[origin=c]{90}{\bf Kolmogorov} &
\begin{tabular}{@{\hspace{0mm}}l@{\hspace{1mm}}|l@{\hspace{1mm}}|c@{\hspace{2mm}}c@{\hspace{2mm}}c@{\hspace{1mm}}|} 
\multicolumn{2}{c}{} & \multicolumn{3}{c}{Actual} \\ \cline{3-5}
\multicolumn{2}{c|}{} & NM & HO & LO \\ \cline{2-5}
\multirow{3}{*}{\rotatebox[origin=c]{90}{Predicted}} & NM & 48 & 23 & 1 \\
 & HO & 25 & 299 & 20 \\
 & LO & 0 & 51 & 623 \\ \cline{2-5}
\end{tabular} &
\begin{tabular}{@{\hspace{0mm}}l@{\hspace{1mm}}|l@{\hspace{1mm}}|c@{\hspace{2mm}}c@{\hspace{2mm}}c@{\hspace{1mm}}|} 
\multicolumn{2}{c}{} & \multicolumn{3}{c}{Actual} \\ \cline{3-5}
\multicolumn{2}{c|}{} & NM & HO & LO \\ \cline{2-5}
\multirow{3}{*}{\rotatebox[origin=c]{90}{Predicted}} & NM & 50 & 26 & 1 \\
 & HO & 23 & 292 & 27 \\
 & LO &  0 &  55 & 616 \\ \cline{2-5}
\end{tabular} &
\begin{tabular}{@{\hspace{0mm}}l@{\hspace{1mm}}|l@{\hspace{1mm}}|c@{\hspace{2mm}}c@{\hspace{2mm}}c@{\hspace{1mm}}|} 
\multicolumn{2}{c}{} & \multicolumn{3}{c}{Actual} \\ \cline{3-5}
\multicolumn{2}{c|}{} & NM & HO & LO \\ \cline{2-5}
\multirow{3}{*}{\rotatebox[origin=c]{90}{Predicted}} & NM & 48 & 22 & 1 \\
 & HO & 25 & 308 & 28 \\
 & LO &  0 &  43 & 615 \\ \cline{2-5}
\end{tabular}\\ [10 mm]
\rotatebox[origin=c]{90}{\bf Wasserstein} &
\begin{tabular}{@{\hspace{0mm}}l@{\hspace{1mm}}|l@{\hspace{1mm}}|c@{\hspace{2mm}}c@{\hspace{2mm}}c@{\hspace{1mm}}|} 
\multicolumn{2}{c}{} & \multicolumn{3}{c}{Actual} \\ \cline{3-5}
\multicolumn{2}{c|}{} & NM & HO & LO \\ \cline{2-5}
\multirow{3}{*}{\rotatebox[origin=c]{90}{Predicted}} & NM & 55 & 21 & 2 \\
 & HO & 18 & 297 & 18 \\
 & LO & 0 & 55 & {\bf 624} \\ \cline{2-5}
\end{tabular} &
\begin{tabular}{@{\hspace{0mm}}l@{\hspace{1mm}}|l@{\hspace{1mm}}|c@{\hspace{2mm}}c@{\hspace{2mm}}c@{\hspace{1mm}}|} 
\multicolumn{2}{c}{} & \multicolumn{3}{c}{Actual} \\ \cline{3-5}
\multicolumn{2}{c|}{} & NM & HO & LO \\ \cline{2-5}
\multirow{3}{*}{\rotatebox[origin=c]{90}{Predicted}} & NM & 54 & 26 & 2 \\
 & HO & 18 & 283 & 20 \\
 & LO &  1 & 64 & {\bf 622} \\ \cline{2-5}
\end{tabular} &
\begin{tabular}{@{\hspace{0mm}}l@{\hspace{1mm}}|l@{\hspace{1mm}}|c@{\hspace{2mm}}c@{\hspace{2mm}}c@{\hspace{1mm}}|} 
\multicolumn{2}{c}{} & \multicolumn{3}{c}{Actual} \\ \cline{3-5}
\multicolumn{2}{c|}{} & NM & HO & LO \\ \cline{2-5}
\multirow{3}{*}{\rotatebox[origin=c]{90}{Predicted}} & NM & 54 & 22 & 2 \\
 & HO & 19 & 295 & 21 \\
 & LO &  0 &  56 & {\bf 621} \\ \cline{2-5}
\end{tabular} \\ [10 mm]
\rotatebox[origin=c]{90}{\bf Zolotarev} &
\begin{tabular}{@{\hspace{0mm}}l@{\hspace{1mm}}|l@{\hspace{1mm}}|c@{\hspace{2mm}}c@{\hspace{2mm}}c@{\hspace{1mm}}|} 
\multicolumn{2}{c}{} & \multicolumn{3}{c}{Actual} \\ \cline{3-5}
\multicolumn{2}{c|}{} & NM & HO & LO \\ \cline{2-5}
\multirow{3}{*}{\rotatebox[origin=c]{90}{Predicted}} & NM & {\bf 59} & 15 & 2 \\
 & HO & 14 & {\bf 302} & 20 \\
 & LO & 0 & 56 & 622 \\ \cline{2-5}
\end{tabular} &
\begin{tabular}{@{\hspace{0mm}}l@{\hspace{1mm}}|l@{\hspace{1mm}}|c@{\hspace{2mm}}c@{\hspace{2mm}}c@{\hspace{1mm}}|} 
\multicolumn{2}{c}{} & \multicolumn{3}{c}{Actual} \\ \cline{3-5}
\multicolumn{2}{c|}{} & NM & HO & LO \\ \cline{2-5}
\multirow{3}{*}{\rotatebox[origin=c]{90}{Predicted}} & NM & {\bf 57} & 15 & 3 \\
 & HO & 16 & {\bf 317} & 33 \\
 & LO & 0 & 41 & 608 \\ \cline{2-5}
\end{tabular} &
\begin{tabular}{@{\hspace{0mm}}l@{\hspace{1mm}}|l@{\hspace{1mm}}|c@{\hspace{2mm}}c@{\hspace{2mm}}c@{\hspace{1mm}}|} 
\multicolumn{2}{c}{} & \multicolumn{3}{c}{Actual} \\ \cline{3-5}
\multicolumn{2}{c|}{} & NM & HO & LO \\ \cline{2-5}
\multirow{3}{*}{\rotatebox[origin=c]{90}{Predicted}} & NM & {\bf 60} & 15 & 2 \\
 & HO & 13 & {\bf 307} & 21 \\
 & LO & 0 & 51 & {\bf 621} \\ \cline{2-5}
\end{tabular}
\end{tabular}
\end{center}
\end{table}
}

Of the three distances, $\kappa$, $\bar\omega$ and $\bar\zeta_2$, for the next diagram we have kept only the latter, as it has the highest discriminant ability. Figure~\ref{Figure.COUPLogMedEnLogZol} displays the scatterplot of the logarithm of MedEn in terms of the logarithm of the normalized Zolotarev distance to the exponential distribution.
At a glance, we see the three COUP classes separated in clusters (extragalactic: high MedEn/low distance; heavily obscured: high energy/high distance; lightly obscured: low energy/medium distance).

In Figure~\ref{Figure.COUPLogMedEnLogZol}, a simple visual inspection reveals that COUP sources 751, 895 and 961, classified as extragalactic, have a higher probability of being heavily obscured stars. Indeed, Table~\ref{TableCOUPPosteriorProbs} displays the posterior probabilities of membership, derived from the quadratic classification rule, for the misclassified extragalactic cases when these probabilities exceed 0.9. We see that all the sources in Table~\ref{TableCOUPPosteriorProbs} are actually the largest outliers ``detected'' by the normalized Zolotarev distance $\bar\zeta_2$ in the extragalactic class (Figure~\ref{Figure.BoxplotsCOUP}(c)). This emphasizes the information conveyed by $\bar\zeta_2$ on the source class.

\begin{figure}
\begin{center}
\includegraphics[width=\textwidth]{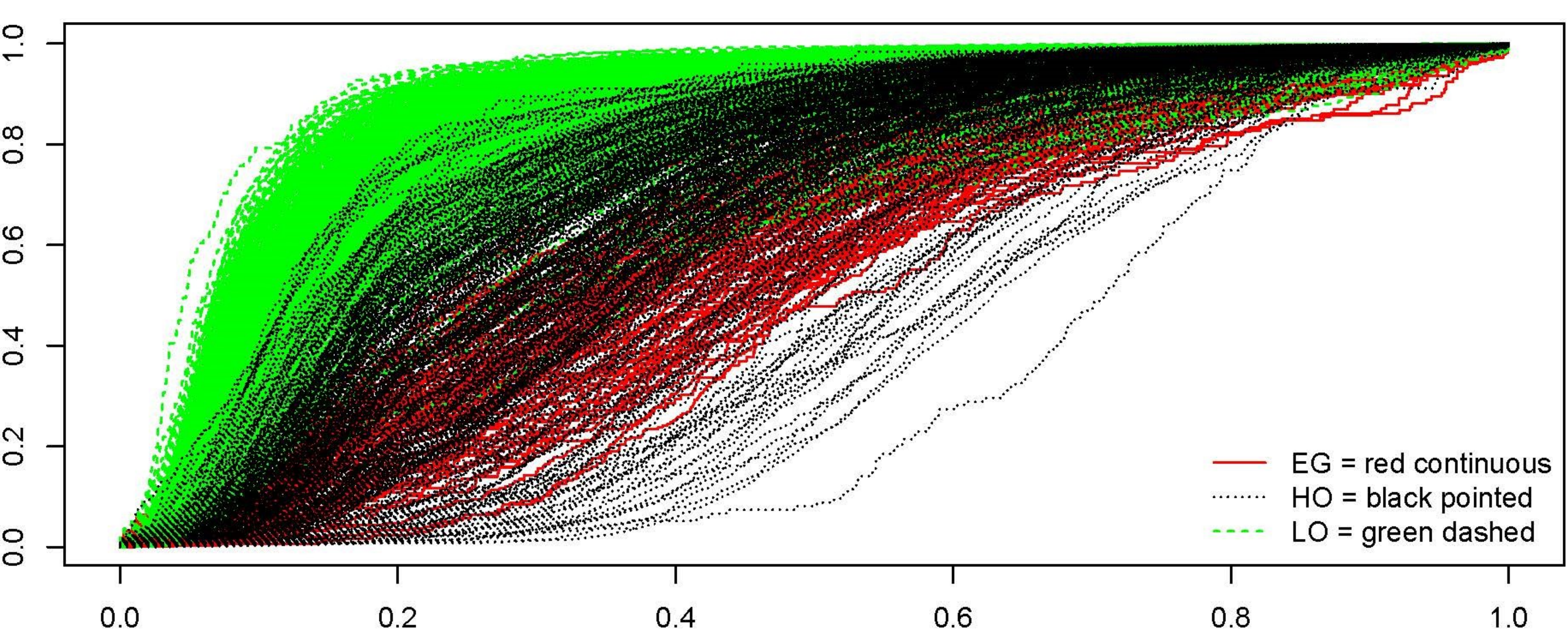}
\end{center}
\caption{Analysis of COUP data: Empirical distribution functions of the photon energies rescaled to the interval [0,1].}
\label{Figure.COUPEnergiesDistribFns}
\end{figure}

\begin{figure}
\begin{center}
\includegraphics[width=\textwidth]{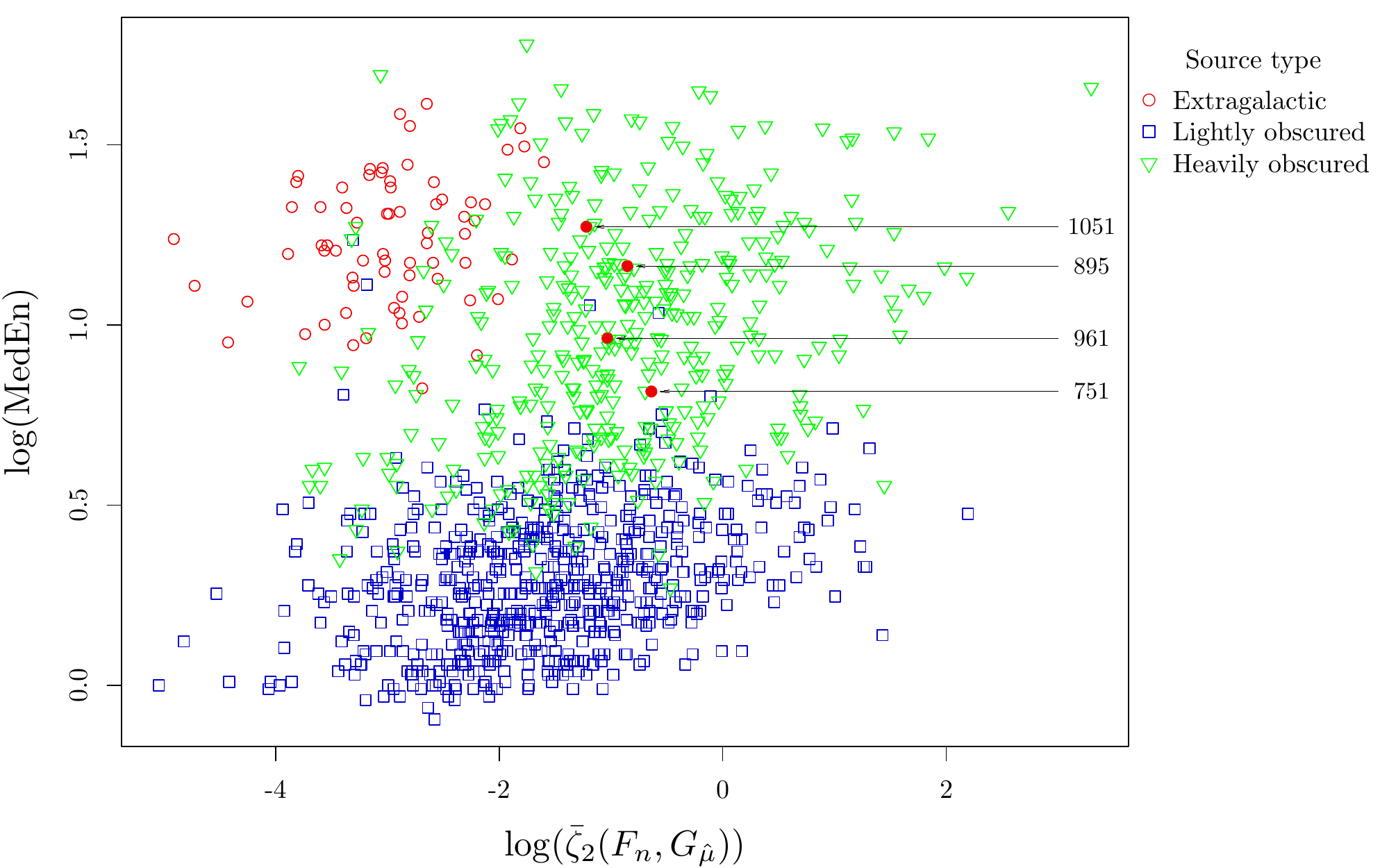}
\end{center}
\caption{Analysis of COUP data: Scatterplot of the logarithm of the median energy and the logarithm of the normalized Zolotarev distance to the exponential distribution. The four outlying sources of the extragalactic class, identified by the Zolotarev distance, are highlighted with filled circles.}
\label{Figure.COUPLogMedEnLogZol}
\end{figure}

\begin{table}
\caption{Analysis of COUP data: Posterior probabilities of membership for extragalactic outliers.}
\label{TableCOUPPosteriorProbs}
\begin{center}
\begin{tabular}{rccc}
 & \multicolumn{3}{c}{Posterior probabilities (in percentage)} \\ \cline{2-4}
\multicolumn{1}{c}{Source} & Extragalactic & Heavily obs.  & Lightly obs. \\ \hline
751 & 0.0152 & 92.8997 & 7.08514 \\ 
895 & 1.5713 & 98.4230 & 0.00576 \\ 
961 & 0.6601 & 98.9164 & 0.42343 \\ 
1015 & 8.4486 & 91.5511 & 0.00003 \\ 
\end{tabular}
\end{center}
\end{table}


\section{Conclusions}\label{SectionConclusions}

We have introduced normalized versions of two integral probability metrics, the Wasserstein and Zolotarev distances, to quantify the discrepancy between photon interarrival times of X-ray cosmic sources and the exponential distribution. The aim is to measure how different the photon emitting source is from extragalactic radiation. The plug-in estimators of these metrics show a good asymptotic behaviour. The analysis of more than one thousand X-ray sources from the Chandra Orion Ultradeep Project with the proposed metrics reveals that the information conveyed by photon interarrival times on the nature of each X-ray source is very well summarized by the normalized Zolotarev distance. We have further shown that this metric, together only with the median energy of the X-ray source, yields a high percentage of correct classifications of the sources into the classes previously provided by astrophysicists. We remark that here we have only used two discriminating features while the usual expert procedures in astronomy rely on many more. As a striking conclusion, we have detected four sources, originally classified as extragalactic, which, as a matter of fact, are very likely young stars in Orion Molecular Cloud 1.


\section*{Appendix}

In this technical appendix, we collect the proofs of the results in Section \ref{Section Zolotarev.Subsection Asymptotics}. The main ingredients are the following: first, we show that the sequences of stochastic processes defined in (\ref{processes X(d)}) are equivalent in $L^1$ to continuous functionals of the empirical process; then we apply the central limit theorem (CLT) in suitable Banach spaces to find their weak limits; finally, we use the continuity of the functional given in (\ref{rhon}) to derive the asymptotic distribution of $\delta_n(d,F)$ in (\ref{normalized and estimated distance}).

To begin, we recall that if the stochastic processes  $\mathbb{P}_n$ and $\mathbb{P}$ take values in $L^1$ a.s., it is said that $\mathbb{P}_n$ \emph{converges in distribution to $\mathbb{P}$ in $L^1$} if $\lim_{n\to\infty} \E f(\mathbb{P}_n)=\E f(\mathbb{P})$, for all continuous and bounded functions $f:L^1\longrightarrow \R$.
Note that  if $\mathbb{P}$ and $\mathbb{P}_n$ are jointly measurable and have almost all their trajectories in $L^1$, they can be identified with Borel-measurable random elements in $L^1$. Therefore, the previous expectations are well-defined. In the following, we denote this weak convergence of probability measures in $L^1$ by $\mathbb{P}_n \cw \mathbb{P}$. An analogous definition can be given for the weak convergence in the weighted $L^1$ space defined by
\begin{equation*}
W^1:=\left\{ f\in L^1: \Vert f\Vert_{W^1}:=\int_0^\infty (1+t)|f(t)|\,  \dif t<\infty\right\}.
\end{equation*}

Let  $\mathbb{P}_n$ and $\widetilde{\mathbb{P}}_n$ be two stochastic processes with trajectories in $L^1$ a.s. We say that $\mathbb{P}_n$ and $\widetilde{\mathbb{P}}_n$ are \textit{equivalent in $L^1$}, denoted by $\mathbb{P}_n\aequiv \widetilde{\mathbb{P}}_n$, if $\Vert \mathbb{P}_n-\widetilde{\mathbb{P}}_n\Vert_1\cp 0$, where ``$\cp$" stands for convergence in probability. Roughly speaking, if $\mathbb{P}_n\aequiv \widetilde{\mathbb{P}}_n$, the two processes have the same asymptotic behavior in $L^1$ because if $\mathbb{P}_n\cw \mathbb{P}$ and $\mathbb{P}_n\aequiv \widetilde{\mathbb{P}}_n$, then $\widetilde{\mathbb{P}}_n\cw \mathbb{P}$ (see for instance \cite[Theorem 18.10]{vanderVaart}).

In the sequel, $\Emp_n$ stands for the \emph{empirical process} associated to $X$, that is,
$\Emp_n(t):=\sqrt{n}(F_n(t)-F(t))$, $t\ge 0$, $n\ge 1$.
The asymptotic behavior of $\Emp_n$ in $L^1$ and $W^1$ is collected in the following lemma. Part (a) is a known result (see  \cite[Theorem 2.1]{delBarrio}), while part (b) can be found in \cite[Lemma 2]{Baillo-Carcamo-Nieto}.

\begin{lemma}\label{Lemma Convergence of E}
We have that
\begin{enumerate}
\item[(a)] $\Emp_n\cw \B_F$ if and only if $X\in \mathcal{L}^{2,1}$.

\item[(b)] $\Emp_n\overset{W^1}{\longrightarrow}_{\text{w}} \B_F$ if and only if $X\in \mathcal{L}^{4,2}$.

\end{enumerate}
\end{lemma}

Our first task is to find processes expressed as continuous functionals of $\Emp_n$ and equivalent to $\mathbb{X}_{d,n}$ in $L^1$. To start, for $t\ge 0$, we decompose $\mathbb{X}_{d,n}$ given in (\ref{processes X(d)}) as:
\begin{equation}\label{decomposition}
\mathbb{X}_{d,n}=\mathbb{A}_{d,n}+\mathbb{B}_{d,n}+\mathbb{C}_{d,n},
\end{equation}
where
\begin{align*}
&\mathbb{A}_{\omega,n}:=\Emp_n,  &  &\mathbb{B}_{\omega,n}:= \sqrt{n}(G_\mu-G_{\hat\mu}),  &  &\mathbb{C}_{\omega,n}:= 0, \\
&\mathbb{A}_{\bar\omega,n}:=\Emp_n/\hat\mu,  &  &\mathbb{B}_{\bar\omega,n}:= \sqrt{n}(G_\mu-G_{\hat\mu})/\hat\mu,  &  &\mathbb{C}_{\bar\omega,n}:= g_{\bar\omega} \sqrt{n}(\mu-\hat\mu)/\hat\mu, \\
&\mathbb{A}_{\zeta_2,n}(t):=\int_t^\infty \Emp_n,  &  &\mathbb{B}_{\zeta_2,n}(t):= \sqrt{n}\int_t^\infty(G_\mu-G_{\hat\mu}),  &  &\mathbb{C}_{\zeta_2,n}:= 0, \\
&\mathbb{A}_{\bar\zeta_2,n}(t):=\frac{1}{\hat\mu^2}\int_t^\infty \Emp_n,  &  &\mathbb{B}_{\bar\zeta_2,n}(t):= \frac{\sqrt{n}}{\hat\mu^2}\int_t^\infty(G_\mu-G_{\hat\mu}),  &  &\mathbb{C}_{\bar\zeta_2,n}:= g_{\bar\zeta_2}\frac{\sqrt{n}(\mu^2-\hat\mu^2)}{\hat\mu^2},
\end{align*}
with $g_{\bar\omega}$ and $g_{\bar\zeta_2}$ defined in (\ref{funtions omega}) and (\ref{funtions zeta}), respectively.

Lemma~\ref{lemma equivalence} provides equivalent expressions for the processes defined above.

\begin{lemma}\label{lemma equivalence}
Let $X$ be a positive random variable with mean $\mu>0$. For $t\ge 0$, the following assertions hold:
\begin{enumerate}
\item[(a)] If $X\in\mathcal{L}^{4/3}$,  $\mathbb{B}_{\omega,n}\aequiv \widetilde{\mathbb{B}}_{\omega,n}$, where $\widetilde{\mathbb{B}}_{\omega,n}(t):=\sqrt{n}(\hat\mu-\mu) t e^{-t/\mu}/\mu^2$.
\item[(b)] If $X\in\mathcal{L}^{2,1}$, $\mathbb{A}_{\bar\omega,n} \aequiv \widetilde{\mathbb{A}}_{\bar\omega,n} := \Emp_n/\mu$.
\item[(c)] If $X\in\mathcal{L}^{4/3}$, $\mathbb{B}_{\bar\omega,n}\aequiv \widetilde{\mathbb{B}}_{\bar\omega,n}$, where $\widetilde{\mathbb{B}}_{\bar\omega,n}(t):=\sqrt{n}(\hat\mu-\mu) t e^{-t/\mu}/\mu^3$.
\item[(d)]  If $X\in\mathcal{L}^{4/3}$, $\mathbb{C}_{\bar\omega,n} \aequiv \widetilde{\mathbb{C}}_{\bar\omega,n} := \sqrt{n}(\mu-\hat\mu) g_{\bar\omega}/\mu $.
\item[(e)]  If $X\in\mathcal{L}^{4/3}$,  $\mathbb{B}_{\zeta_2,n}\aequiv \widetilde{\mathbb{B}}_{\zeta_2,n}$, where $\widetilde{\mathbb{B}}_{\zeta_2,n}(t):=\sqrt{n}(\hat\mu-\mu)  (1+t/\mu)e^{-t/\mu} $.
\item[(f)]   If $X\in\mathcal{L}^{4,2}$, $\mathbb{A}_{\bar\zeta_2,n}\aequiv \widetilde{\mathbb{A}}_{\bar\zeta_2,n}$, where $\widetilde{\mathbb{A}}_{\bar\zeta_2,n}(t):=\int_t^\infty \Emp_n/\mu^2$.
\item[(g)] If $X\in\mathcal{L}^{4/3}$,  $\mathbb{B}_{\bar\zeta_2,n}\aequiv \widetilde{\mathbb{B}}_{\bar\zeta_2,n}$, where $\widetilde{\mathbb{B}}_{\bar\zeta_2,n}(t):=\sqrt{n}(\hat\mu-\mu)   (1+t/\mu)e^{-t/\mu}/\mu^2 $.
\item[(h)] If $X\in\mathcal{L}^{2}$, $\mathbb{C}_{\bar\zeta_2,n} \aequiv \widetilde{\mathbb{C}}_{\bar\zeta_2,n} := \sqrt{n}(\mu-\hat\mu) 2\, g_{\bar\zeta_2} /\mu $.
\end{enumerate}
\end{lemma}

\begin{proof}
To show part (a), we use the mean value theorem twice to obtain
\begin{align}
\Vert \mathbb{B}_{\omega,n}-\widetilde{\mathbb{B}}_{\omega,n}\Vert_1
& \le \sqrt{n}(\hat\mu-\mu)^2 \int_0^\infty t |2-t/\xi_t| e^{-t/\xi_t}/\xi_t^3\, \dif t,\label{integral mean value}
\end{align}
where 
$\xi_t$ is a point between $\mu$ and $\hat\mu$. The integral in (\ref{integral mean value}) is bounded by
\begin{equation}\label{bound integral}
\int_0^\infty t |2-t/\xi_t| e^{-t/\xi_t}/\xi_t^3\, \dif t\le 2 (\mu+\hat\mu)\frac{\max\{\mu,\hat\mu\}^2}{\min\{\mu,\hat\mu\}^4}\to 4/\mu\quad \text{a.s.}
\end{equation}
Therefore, from (\ref{integral mean value})-(\ref{bound integral}), and by the Kolmogorov, Marcinkiewicz and Zygmund strong law of large numbers  (see, e.g., \cite[Theorem 3.23]{Kallemberg}), we see that, whenever $X\in\mathcal{L}^{4/3}$, $\Vert \mathbb{B}_{\omega,n}-\widetilde{\mathbb{B}}_{\omega,n}\Vert_1\to 0$ a.s.

To see (b), we note that $\Vert \mathbb{A}_{\bar\omega,n}-\widetilde{\mathbb{A}}_{\bar\omega,n}\Vert_1=\Vert \Emp_n \Vert_1|\mu-\hat\mu|/ (\mu\hat\mu)$. From Lemma \ref{Lemma Convergence of E} (a), we have that $\Vert \Emp_n \Vert_1 \cd  \Vert \B_F \Vert_1 $ and the conclusion follows from the strong law of large numbers and Slutsky's theorem.

Part (c) follows from (a), as it is straightforward to check that $\mathbb{B}_{\bar\omega,n}\aequiv {{\mathbb{B}}}_{\omega,n}/\mu$, whenever $X\in\mathcal{L}^{4/3}$.

Part (d) is direct, whereas part (e) can be found in \cite[Lemma 1]{Baillo-Carcamo-Nieto}. The proof of part (f) is similar to the one for (b) by using Lemma \ref{Lemma Convergence of E} (b).

To show part (g), we observe that, from part (e), we have that $\widetilde{\mathbb{B}}_{\bar\zeta_2,n} \aequiv \mathbb{B}_{\zeta_2,n}/\mu^2$. The conclusion follows by checking that
$\Vert \mathbb{B}_{\bar\zeta_2,n}- \mathbb{B}_{\zeta_2,n}/\mu^2\Vert_1=\sqrt{n}(\mu-\hat\mu)^2 (1/\mu+1/\hat\mu)^2$.

Finally, it can be seen that
$\Vert \mathbb{C}_{\bar\zeta_2,n}- \widetilde{\mathbb{C}}_{\bar\zeta_2,n}\Vert_1=\sqrt{n}(\mu-\hat\mu)^2(1/\mu+1/\hat\mu)\, \bar\zeta_2(X,Y_\mu)$.
As $\bar\zeta_2(X,Y_\mu)<\infty$ if and only if $X\in\mathcal{L}^2$, we conclude that (h) is fulfilled. \hfill $\Box$
\end{proof}

The next corollary, which is a consequence of Lemma~\ref{lemma equivalence} and (\ref{decomposition}), shows that $\mathbb{X}_{d,n}$ are equivalent in $L^1$ to certain continuous functionals of the empirical process $\Emp_n$.

\begin{corollary}\label{Corollary equivalence}
Let $X$ be a positive random variable with mean $\mu>0$. 
\begin{enumerate}
\item[(i)] If $X\in\mathcal{L}^{4/3}$, then $\mathbb{X}_{\omega,n}\aequiv \phi_\omega(\Emp_n)$, where $\phi_\omega:L^1\to L^1$  is the linear operator 
\begin{equation*}
\phi_\omega(f,t):=f(t)- \frac{t}{\mu^2} e^{-t/\mu} \int_0^\infty f(x)\,\dif x, \quad t\ge 0.
\end{equation*}
Moreover, $\Vert \phi_\omega(f)\Vert_1\le 2\, \Vert f\Vert_1$, and $\phi_\omega$ is therefore continuous.
\item[(ii)] If $X\in\mathcal{L}^{2,1}$, then $\mathbb{X}_{\bar\omega,n} \aequiv \phi_{\bar\omega}(\Emp_n)$, where $\phi_{\bar\omega}:L^1\to L^1$ is the linear operator 
\begin{equation*}
\phi_{\bar\omega}(f,t):=\frac{1}{\mu} \left[f(t)+ \left(g_{\bar\omega}(t)-\frac{t}{\mu^2}e^{-t/\mu}      \right) \int_0^\infty f(x)\,\dif x \right],\quad t\ge 0.
\end{equation*}
Moreover, $\Vert \phi_{\bar\omega}(f)\Vert_1\le \Vert f\Vert_1 (2+\bar{\omega}(X,Y_\mu))/\mu$, and $\phi_{\bar\omega}$ is therefore continuous.
\item[(iii)] If $X\in\mathcal{L}^{4/3}$, then $\mathbb{X}_{\zeta_2,n}\aequiv \phi_{\zeta_2}(\Emp_n)$, where $\phi_{\zeta_2}:W^1\to L^1$ is the linear operator 
\begin{equation*}
\phi_{\zeta_2}(f,t):=\int_t^\infty f(x)\,\dif x- \left(1+\frac{t}{\mu} \right)e^{-t/\mu} \int_0^\infty f(x)\,\dif x,\quad t\ge 0.
\end{equation*}
Moreover, $\Vert \phi_{\zeta_2}(f)\Vert_1\le (1+2\mu)\,\Vert f\Vert_{W^1}$, and $\phi_{\zeta_2}$ is therefore continuous.
\item[(iv)] If $X\in\mathcal{L}^{4,2}$, then  $\mathbb{X}_{\bar\zeta_2,n} \aequiv \phi_{\bar\zeta_2}(\Emp_n)$, where $\phi_{\bar\zeta_2}:W^1\to L^1$ is the linear operator 
\begin{equation*}
\phi_{\bar\zeta_2}(f,t) := \frac{1}{\mu^2} \left[ \int_t^\infty f(x)\,\dif x+ \left(2\,\mu\, g_{\bar\zeta_2}(t)- \left(1+\frac{t}{\mu}\right) e^{-t/\mu}\right) \int_0^\infty f(x)\,\dif x \right],
\end{equation*}
for $t\ge 0$. Moreover, $\Vert \phi_{\bar\zeta_2}(f)\Vert_1\le \Vert f\Vert_{W^1}\,  [1+2\mu(1+\bar\zeta_2(X,Y_\mu))]/\mu^2 $, and $\phi_{\bar\zeta_2}$ is therefore continuous.
\end{enumerate}
\end{corollary}

We are now in condition to prove Theorems \ref{Theorem SP Wasserstein} and \ref{Theorem SP Zolotarev}.
\medskip

\noindent
\textit{Proof of Theorem \ref{Theorem SP Wasserstein}} Assume that (a) holds, i.e., $X\in\mathcal{L}^{2,1}$. For $d=\omega$ or $d=\bar\omega$, from Lemma \ref{Lemma Convergence of E} (a) and  Corollary \ref{Corollary equivalence} (i) and (ii), we have that $\mathbb{X}_{d,n}\aequiv \phi_d(\Emp_n)\cw \phi_d(\B_F)=\mathbb{X}_{d,F}$, by the continuos mapping theorem. We conclude that (b) and (c) hold.

Conversely, let us assume  that (b) is satisfied. By Corollary \ref{Corollary equivalence} (i), if $X\in\mathcal{L}^{4/3}$, we obtain that $\phi_\omega(\Emp_n)\cw \mathbb{X}_{\omega,F}$. Observe now that $\phi_\omega(\Emp_n)$ can be rewritten as the normalized sum
$\phi_\omega(\Emp_n) = n^{-1/2}\sum_{i=1}^n  \mathbb{Y}_{\omega,i}$
where $ \mathbb{Y}_{\omega,1},\dots,\mathbb{Y}_{\omega,n}$ are $n$ independent copies of the zero mean process
\begin{equation}\label{Process Y(omega)}
\mathbb{Y}_\omega(t):=\Prob(X>t)-I_{\{X>t\}}+(X-\mu)te^{-t/\mu}/\mu^2 ,\quad t\ge 0.
\end{equation}
This means that the process $\mathbb{Y}_\omega$ satisfies the CLT in $L^1$ (and implies that $\mathbb{X}_{\omega,F}$ is a centered Gaussian process), which is equivalent (see \cite[p. 205]{Araujo-Gine}) to
\begin{equation}\label{CLT Y(omega)}
\int_0^\infty \sqrt{\E \mathbb{Y}_\omega(t)^2}\, \dif t<\infty.
\end{equation}
In particular, this implies that $X\in \mathcal{L}^2$ and denoting $\mathbb{Z}(t):=\Prob(X>t)-I_{\{X>t\}}$ ($t\ge 0$), from (\ref{Process Y(omega)}), (\ref{CLT Y(omega)}), and by Minkowski inequality, we have that
$$\int_0^\infty \sqrt{\E \mathbb{Z}(t)^2}\, \dif t\le \sigma+ \int_0^\infty \sqrt{\E \mathbb{Y}_\omega(t)^2}\, \dif t<\infty,$$
where $\sigma$  is the standard deviation of $X$. Last inequality amounts to (a) $X\in\mathcal{L}^{2,1}$.

For the proof that (c) implies (a), it is enough to note that $\mathbb{X}_{\bar\omega,n}\cw \mathbb{X}_{\bar\omega,F}$ is equivalent to $\hat\mu\,\mathbb{X}_{\bar\omega,n}\cw \mu\,\mathbb{X}_{\bar\omega,F}$. The rest of the proof runs as with $d=\omega$. \hfill $\Box$

\medskip

\noindent
\textit{Proof of Theorem \ref{Theorem SP Zolotarev}}  To show that part (a) implies (b) and (c) it is enough to follow the same steps as in the proof of the same implications in Theorem \ref{Theorem SP Wasserstein}. We omit the details.

To finish, we will show that part (c) implies (a) (the remaining implication ``(b) $\Rightarrow$ (a)" is simpler and similar). Let us assume that (c) is satisfied. In this situation, it is clear that $X\in \mathcal{L}^2$, as this integrability condition amounts to saying that the process $\mathbb{X}_{\bar\zeta_2,n}$ has its paths in $L^1$ a.s. We have that $\hat\mu^2\mathbb{X}_{\bar\zeta_2,n}\cw \mu^2 \mathbb{X}_{\bar\zeta_2,F}$. Further, by Lemma \ref{lemma equivalence}, we conclude that $\hat\mu^2\mathbb{X}_{\bar\zeta_2,n}\aequiv \overline{\mathbb{Z}}_n$, where
\begin{equation*}
\overline{\mathbb{Z}}_n(t):=\int_t^\infty \Emp_n+ h(t) \sqrt{n}(\mu-\hat\mu),\quad t\ge 0,
\end{equation*}
with
\begin{equation}\label{function h}
h(t):=2\mu\, g_{\bar\zeta_2}(t)-(1+t/\mu)e^{-t/\mu},\quad t\ge 0.
\end{equation}
The process $\overline{\mathbb{Z}}_n$ can be rewritten as the normalized sum
$\overline{\mathbb{Z}}_n=n^{-1/2} \sum_{i=1}^n \mathbb{Z}_{\bar\zeta_2,i}$,
where $ \mathbb{Z}_{\bar\zeta_2,1},\dots,\mathbb{Z}_{\bar\zeta_2,n}$ are $n$ independent copies of the zero mean process
\begin{equation*}
\mathbb{Z}_{\bar\zeta_2}(t):=\E(X-t)_+-(X-t)_++h(t)\,(\mu-X).
\end{equation*}
Therefore, the process $\mathbb{Z}_{\bar\zeta_2}$ satisfies the CLT in $L^1$ (and we see that $\mathbb{X}_{\bar\zeta_2,F}$ is a centered Gaussian process). Using again \cite[p. 205]{Araujo-Gine}, we obtain that
\begin{equation}\label{CLT th 2}
\int_0^\infty \sqrt{\E \mathbb{Z}_{\bar\zeta_2}(t)^2}\, \dif t<\infty.
\end{equation}
Finally, by  Minkowski inequality and Fubini theorem, we have that
\begin{equation*}
\int_0^\infty \sqrt{\E(X-t)_+^2}\, \dif t\le \E X^2/2+\sigma\Vert h\Vert_1+\int_0^\infty \sqrt{\E \mathbb{Z}_{\bar\zeta_2}(t)^2}\, \dif t.
\end{equation*}
From (\ref{function h}), we can check that $\Vert h\Vert_1\le 2\mu(1+\bar\zeta_2(X,Y_\mu))<\infty$, and from (\ref{CLT th 2}), we conclude that $\int_0^\infty \sqrt{\E(X-t)_+^2}\, \dif t<\infty$. This implies that $X\in\mathcal{L}^{4,2}$ because $t^2\Prob(X>2t)\le \E(X-t)_+^2$.
\hfill $\Box$

\medskip




\noindent
\textit{Proof of Theorem \ref{Theorem 3}} We have that $\delta_n(d,F)=\rho_n(\mathbb{X}_{d,n},g_d)$, with $\rho_n$ defined in (\ref{rhon}).
The continuity of $\rho_n$ (in $L^1$) was analyzed in \cite[Lemma 4]{Carcamo}, where it was shown that if $f_n\to f$ in $L^1$ and $g\in L^1$, then, for $I(g):=\{t\ge 0: g(t)=0\}$,
\begin{equation}\label{rho}
\lim_{n\to\infty} \rho_n(f_n,g) = \rho(f,g):=\int_{I(g)} |f|+\int_{I(g)^c}f\, \sgn(g).
\end{equation}

Therefore, from Theorems \ref{Theorem SP Wasserstein} and \ref{Theorem SP Zolotarev}, (\ref{rho}), and the extended continuous mapping theorem (see \cite[Theorem 1.11.1]{vanderVaart-Wellner}), we conclude that
$\delta_n(d,F)=\rho_n(\mathbb{X}_{d,n},g_d) \cd  \rho(\mathbb{X}_{d,F},g_d)=\delta_\infty(d,F)$, as $n\to\infty$.
\hfill $\Box$

\medskip

\noindent
\textit{Proof of Corollary \ref{Corrolary normal}} From Theorem \ref{Theorem 3}, we have that
\begin{equation*}
\delta_\infty(d,F)=\int_0^\infty \mathbb{X}_{d,F}(t)\, \sgn(g_d(t))\,\dif t.
\end{equation*}
As $\mathbb{X}_{d,F}$ is a centered Gaussian process and $g_d$ is nonrandom, we conclude that $\delta_\infty(d,F)$ is normally distributed.
\hfill $\Box$


\section*{Acknowledgements}

The authors are grateful to three reviewers and the associate editor for their insightful comments
which have improved the presentation of the paper.

\end{document}